\documentclass[12pt]{article}
\usepackage{amsmath,amsfonts,amssymb,amsthm}
\begin{document}

\newcommand{\1}{{{\bf 1}}}
\newcommand{\id}{{\rm id}}
\newcommand{\Hom}{{\rm Hom}\,}
\newcommand{\End}{{\rm End}\,}
\newcommand{\Res}{{\rm Res}\,}
\newcommand{\Image}{{\rm Im}\,}
\newcommand{\Ind}{{\rm Ind}\,}
\newcommand{\Aut}{{\rm Aut}\,}
\newcommand{\Ker}{{\rm Ker}\,}
\newcommand{\gr}{{\rm gr}}
\newcommand{\Der}{{\rm Der}\,}
\newcommand{\Span}{{\rm Span}}
\newcommand{\tr}{{\rm tr}}

\newcommand{\Z}{\mathbb{Z}}
\newcommand{\Q}{\mathbb{Q}}
\newcommand{\C}{\mathbb{C}}
\newcommand{\N}{\mathbb{N}}

\newcommand{\g}{\mathfrak{g}}
\newcommand{\h}{\mathfrak{h}}
\newcommand{\wt}{{\rm wt}\;}
\newcommand{\CR}{\mathcal{R}}
\newcommand{\D}{\mathcal{D}}
\newcommand{\E}{\mathcal{E}}
\newcommand{\Lie}{\mathcal{L}}
\newcommand{\z}{\bf{z}}
\newcommand{\bflam}{\bf{\lambda}}

\newcommand{\levelk}{\bf{k}}

\newcommand \be{\begin{equation}\label}
\newcommand \ee{\end{equation}}
\newcommand \bex{\begin{exa}\label}
\newcommand \eex{\end{exa}}
\newcommand \bl{\begin{lem}\label}
\newcommand \el{\end{lem}}
\newcommand \bt{\begin{thm}\label}
\newcommand \et{\end{thm}}
\newcommand \bp{\begin{prop}\label}
\newcommand \ep{\end{prop}}
\newcommand \br{\begin{rem}\label}
\newcommand \er{\end{rem}}
\newcommand \bc{\begin{coro}\label}
\newcommand \ec{\end{coro}}
\newcommand \bd{\begin{de}\label}
\newcommand \ed{\end{de}}

\newcommand \bproof{\noindent \textbf{Proof:\ \ }}
\newcommand \eproof{$\blacksquare$}

\newtheorem{thm}{Theorem}[section]
\newtheorem{prop}[thm]{Proposition}
\newtheorem{coro}[thm]{Corollary}
\newtheorem{conj}[thm]{Conjecture}
\newtheorem{exa}[thm]{Example}
\newtheorem{lem}[thm]{Lemma}
\newtheorem{rem}[thm]{Remark}
\newtheorem{de}[thm]{Definition}
\newtheorem{hy}[thm]{Hypothesis}
\makeatletter
\@addtoreset{equation}{section}
\renewcommand{\theequation}{\thesection.\arabic{equation}}
\makeatother
\makeatletter

%\begin{flushright}
%\today
%\end{flushright}
%\baselineskip=24pt

\title{On a symmetry of the category of integrable modules}
\author{William J. Cook \footnotemark[1], Christopher M. Sadowski\footnotemark[2]\\
{\small \footnotemark[1] Department of Mathematical Sciences, 
                         Appalachian State University, Boone, NC 28608}\\
{\small \footnotemark[2] Department of Mathematics, 
                         Rutgers University, Piscataway, NJ 08854}\\}

%%%%this should generate a footnote with no number/symbol%%%%
\renewcommand{\thefootnote}{\fnsymbol{footnote}}
\footnotetext[0]{Sadowski acknowledges support from the the Rutgers Mathematics/DIMACS REU Program during the summers of 2007 and 2008, and NSF grant DMS-0603745.}

%\today
\maketitle 
 
\begin{abstract}
Haisheng Li showed that given a module $(W,Y_W(\cdot,x))$ for a vertex algebra
$(V,Y(\cdot,x))$, one can obtain a new $V$-module $W^{\Delta} = (W,Y_W(\Delta(x)\cdot,x))$ 
if $\Delta(x)$ satisfies certain natural conditions. Li presented a collection of 
such $\Delta$-operators for $V=L(k,0)$ (a vertex operator algebra associated 
with an affine Lie algebras, $k$ a positive integer). In this paper, for each irreducible
$L(k,0)$-module $W$, we find a highest weight vector of $W^{\Delta}$ when $\Delta$ is
associated with a miniscule coweight. From this we completely determine the 
action of these $\Delta$-operators on the set of isomorphism equivalence classes of 
$L(k,0)$-modules.

\vspace{0.1in}


\vspace{0.1in}

\end{abstract}

\section{Introduction} 

Haisheng Li introduced his $\Delta$-operators in a very general setting in \cite{li-local}.
These operators allow one to obtain new vertex algebra modules from old ones by modifying 
the vertex algebra's action on the module while leaving the underlying vector space unchanged.
Thus, given a vertex algebra $V$ and a collection $\Delta$-operators, we obtain (usually quite interesting) symmetries of the category of $V$-modules. 

Consider the $\hat{\g}$-module $L(k,0) = L(k\Lambda_0)$ where $k$ is a positive integer
and $\hat{\g}$ is an untwisted affine Lie algebra. It is well known that $L(k,0)$ has
the structure of a vertex operator algebra (VOA), and that $\hat{\g}$-modules $L(k,\lambda)$ for
certain $\lambda$ are modules for this VOA. Define 
$\displaystyle{\Delta(H,x) = x^{H(0)}\exp \left(\sum_{n\ge 1}\frac{H(n)}{-n}(-x)^{-n}\right)}$
where $H$ is a coweight of $\g$ (the underlying finite dimensional Lie algebra). If 
$(W,Y(\cdot,x))$ is an $L(k,0)$-module, then $(W,Y(\Delta(H,x)\cdot,x))$ (call this 
module $W^{(H)}$) is also an $L(k,0)$-module. In fact, Li proved \cite{li-affine} 
that these two modules are equivalent if $H$ is a coroot. However, when $H$ is not a coroot, 
we may get a new (inequivalent) module. So, using Li's $\Delta$-operators we can induce an 
action on the equivalence classes of $L(k,0)$-modules. In his thesis \cite{c-thesis} 
(see also \cite{clm}), the first author was able to use these operators to 
obtain recurrence relations for characters of integrable highest weight  
$\hat{\g}$-modules which in turn lead to interesting combinatorial identities. These recurrence
relations were obtained by studying the effect of $\Delta(H,x)$ (for $H$ a coroot) on 
characters. If we consider a coweight instead of a coroot, we obtain
a relation between the characters of two {\it different} integrable $\hat{\g}$-modules.
 
In this paper, we completely determine the action of $\Delta(H,x)$ on equivalence classes of $L(k,0)$-modules for all coweights $H$. 

First, recall that $L(k,0)$ is a {\it regular} VOA. This implies (among other things)
that its modules are completely reducible. This means that we just need to determine $\Delta(H,x)$'s action on irreducible $L(k,0)$-modules (these are precisely the integrable 
$\hat{\g}$-modules of level $k$). 

Next, we know that the miniscule coweights provide a complete set of coset representatives 
for the coweight lattice modulo the coroot lattice. So, since the coroots give back 
equivalent modules, it is enough to study the action of $\Delta(H,x)$ on an irreducible
$L(k,0)$-module $L(k,\lambda)$ where $H$ is one of the miniscule coweights. To determine
which module $L(k,\lambda^{(H)})$ is obtained from $L(k,\lambda)$ via the 
$\Delta(H,x)$-action, it is enough to identify a highest weight vector and ``measure'' 
its weight. 

In this paper, we explicitly determine a highest weight vector for $L(k,\lambda)^{(H)}$ 
-- that is, the module $(L(k,\lambda), Y(\Delta(H,x)\cdot,x))$ -- if $L(k,\lambda)$ is an
$L(k,0)$-module and $H$ is a miniscule coweight (of $\g$). 

Recall that the $\hat{\g}$-modules $L(k,\lambda)$ are induced up from $\g$-modules
$L(\lambda)$ and in fact $L(\lambda)$ (the ``finite dimensional part'') makes
up the lowest conformal weight space of $L(k,\lambda)$. Now let us change the action
of $L(k,\lambda)$ from $Y(\cdot,x)$ to $Y(\Delta(H,x)\cdot,x)$ where $H$ is a miniscule
coweight. We get a new $L(k,0)$-module action while leaving the underlying vector space
fixed. It is interesting to note that in each case, the highest weight vector stays inside 
the lowest conformal weight space of $L(k,\lambda)$. Li in \cite{li-affine} considered the
special case, $L(k,0)$, and found that the old highest weight vector (which is the vaccuum
vector) is also the new highest weight vector. This happens because the lowest conformal
weight space is 1 dimensional -- that is, the vector has nowhere to go. On the other
hand, when we consider $L(k,\lambda)$, $\lambda \not=0$, the lowest conformal weight
space is bigger and so the highest weight vector moves and thus is harder to find.

One rather interesting feature of the action of $\Delta(H,x)$ is that the restriction on
$\lambda$ so that $L(k,\lambda)$ is an $L(k,0)$-module shows up explicitly in this action.
We know that if $L(k,\lambda)$ is an $L(k,0)$-module, then 
$\langle    \lambda,\theta \rangle  \leq k$ where $\theta$ is the highest long root of $\g$.
In each case, when acting on $L(k,\lambda)$ with $\Delta(H,x)$ where $H$ is the $j^{th}$
miniscule coweight, we see that the coefficent of $\lambda_j$ (the $j^{th}$ fundamental weight)
is replaced by $k-\langle    \lambda,\theta \rangle $. 

Without making any changes, the calculations and proofs involved in determining the
new highest weight vectors of the $L(k,0)$-modules $L(k,\lambda)^{(H)}$ apply 
to any $V(k,0)$-module as well ($V(k,0)$ is a generalized Verma module which also has 
VOA structure). It is well known that $L(k,\lambda)$ where $\lambda$ is a dominant
integral weight of $\g$ is a $V(k,0)$-module. However, if $\langle    \lambda,\theta \rangle  > k$ (so that $L(k,\lambda)$ is {\it not} an $L(k,0)$-module), then the $\Delta(H,x)$-action 
will produce an (irreducible) {\it weak} $V(k,0)$-module $L(k,\lambda^{(H)})$. But $L(k,\lambda^{(H^{(j)})})$ is not a $V(k,0)$-module (as a 
VOA) since the coffecient of $\lambda_j$ in $\lambda^{(H^{(j)})}$ is negative (and thus not
a dominant integral weight). So for $V(k,0)$-modules, the $\Delta$-action can move (strong) modules to weak modules.

This paper grew out of the second author's summer research experience for undergraduates 
(REU) mentored by the first author and Yi-Zhi Huang at Rutgers University during the summer of 2007. 

It is interesting to note that Li's $\Delta$-operators also allow one to create new intertwining operators from old ones (\cite{li-affine}, Proposition 2.12). In fact, since the $\Delta$-operators are invertible, they give isomorphisms between spaces of intertwining operators. Therefore, using the results above, one can obtain symmetries of fusion rules.
This is the topic of a future project of the authors with Sjuvon Chung and Yi-Zhi Huang.

The authors would like to thank Yi-Zhi Huang for his encouragement and advice throughout this project. 

The contents of the paper are organized as follows:

In the second section, we begin by reviewing the definition of a vertex operator algebra and
its modules. Then, we set up all of the necessary notation related to finite dimensional simple Lie algebras and untwisted affine Lie algebras and conclude by introducing Li's $\Delta$-operators and performing some preliminary calculations. 

The third section illustrates our results in the most basic case -- that of $\hat{\g} = \widehat{sl_2}$. The fourth section tackles the general case where we must consider the 
effects of the $\Delta$-operators one type at a time. In the course of deriving these results,
we need to perform some tedious Weyl group calculations. The calculations for types $B_\ell$,
$C_\ell$, and $D_\ell$ are located in an appendix (the fifth section). The appendix also
contains a summary of results for types $E_6$ and $E_7$. The calculations for types $E_6$
and $E_7$ were performed in Maple using a modified version of a worksheet developed by 
the first author for another project \cite{cms}.

\section{Definitions and Background}

We will begin by reviewing the definition of a vertex operator algebra and its modules.
Further details can be found in \cite{ll-book}.

\subsection{Vertex Algebras}

A {\em vertex algebra} is a complex vector space $V$ equipped with a linear map,
\begin{eqnarray}
Y(\cdot,x):V & \rightarrow & (\End V)[[x,x^{-1}]] \notag \\
           v & \mapsto     & Y(v,x) = \sum_{n \in \Z} v_n x^{-n-1} \notag
\end{eqnarray}
and a distinguished vector $\1 \in V$ (the \emph{vacuum vector})
such that for $u,v\in V$, $u_n v=0$ for $n$ sufficiently large. The operator $u_n$
is called the \emph{$n^{th}$-mode} of $u$.

It is assumed that the vacuum vector behaves like an identity in that 
$Y(\1,x)=\mathrm{Id}_V$ and for $v \in V$
\[ Y(v,x)\1 \in V[[x]] \quad \mathrm{and} \quad 
   \lim_{x \rightarrow 0} Y(v,x)\1 = v \qquad \mathrm{(the\;creation\;axiom)}. \]

Finally, we also must require that the \emph{Jacobi Identity} holds:
for all $u,v \in V$
\begin{eqnarray}
& &x_0^{-1} \delta\left( \frac{x_1 - x_2}{x_0} \right) Y(u,x_1)Y(v,x_2) - 
   x_0^{-1} \delta\left( \frac{x_2 - x_1}{-x_0} \right)
   Y(v,x_2)Y(u,x_1) \nonumber\\
& &\ \ \ \ \ \ 
= x_2^{-1} \delta\left( \frac{x_1 - x_0}{x_2} \right) Y(Y(u,x_0)v,x_2) 
\end{eqnarray}   
Please note that $\delta(x) = \sum_{n \in \Z} x^n$ is the \emph{formal delta function}
and we adopt the binomial expansion convention, namely, $(x+y)^n$ should be expanded
in non-negative powers of $y$.

The vertex algebras that we will consider have additional structure making them
vertex operator algebras. A {\em vertex operator algebra} is a vertex algebra 
$V$ with the following additional data:

$V$ is a $\Z$-graded vector space $V = \coprod_{n \in \Z} V_{(n)}$ (over $\C$) 
such that $\dim V_{(n)}<\infty$ for all $n\in \Z$ and
$V_{(n)} = 0$ for $n$ sufficiently negative. 

Elements of $V_{(n)}$ are said to have \emph{conformal weight} $n$.
The vacuum must have conformal weight 0 (e.g. $\1 \in V_{(0)}$).

$V$ has a second distinguished vector $\omega \in V_{(2)}$ (the \emph{conformal vector}),
where
\begin{equation}
Y(\omega,x) = \sum_{n \in \Z} \omega_n x^{-n-1} = \sum_{n \in \Z} L(n) x^{-n-2} \notag
\end{equation}
The modes of the conformal vector satisfy the \emph{Virasoro} relations: 
\begin{equation} 
[L(m),L(n)] = (m-n)L(m+n) + \frac{m^3-m}{12} \delta_{m+n,0} \, c_V 
            \quad \mathrm{for\;} m,n \in \Z.   
\end{equation}
The scalar $c_V \in \C$ is called the \emph{central charge} (or \emph{rank}) of $V$.
Finally, we must require that
\begin{eqnarray} 
  & & Y(L(-1)v,x) = \frac{d}{dx} Y(v,x) \;\;\; \mbox{for} \; v\in V,\\ \notag
  & & V_{(n)} = \{ v \in V \,|\, L(0)v = nv \} \;\;\; \mbox{for} \; n \in \Z. \notag
\end{eqnarray}

Let $V$ be a vertex algebra. A $V$-module is a complex vector space $W$ equipped
with a linear map
\begin{eqnarray}
Y_W(\cdot,x): V & \rightarrow & (\End W)[[x,x^{-1}]] \notag \\
              v & \mapsto     & Y(v,x) = \sum_{n \in \Z} v_n x^{-n-1} \notag
\end{eqnarray}
such that for $v\in V$ and $w \in W$, $v_n w=0$ for $n$ sufficiently large.
Also, $Y_W(\1,x)=\mathrm{Id}_W$ and the Jacobi identity holds.

If $V$ is a vertex operator algebra, we say $W$ is a (strong) $V$-module if $W$ is a module for $V$ thought of as a vertex algebra and in addition $W$ is
a $\C$-graded vector space $W = \coprod_{n \in \Z} W_{(n)}$
where $W_{(n)} = \{ w \in W \,|\, L(0)w = nw \}$ such that for all $n \in \C$
$\dim W_{(n)} < \infty$ and $W_{(n+r)} = 0$ for $r$ sufficiently negative.

If $W$ is a vertex algebra module for a vertex operator algebra $V$, we say
that $W$ is a \emph{weak} $V$-module.

\subsection{Affine Lie Algebras}

Following \cite{hum} and \cite{kac-book}, we now establish some notation and
review some basic definitions and facts concerning (untwisted) affine Lie algebras.

Let $\g$ be a finite dimensional simple Lie algebra of rank $\ell$ (over $\C$). 
Fix a Cartan subalgebra $\h \subset \g$, and let $\langle   \cdot,\cdot\rangle $ be the standard
form such that $\|\alpha\|^2=\langle   \alpha,\alpha\rangle =2$ for any long root $\alpha$. 
Let $\Delta$ be the set of roots of $\g$. Fix a set of simple roots 
$\{\alpha_1, \dots, \alpha_{\ell} \}$, simple coroots $\{ H_1, \dots, H_{\ell} \}$,
and Chevalley generators $\{E_i,F_i,H_i\,|\,i=1,\dots,\ell\}$.
Recall that $\alpha_j(H_i) = a_{ij}$ where $C=(a_{ij})_{i,j=1}^\ell$ is
the Cartan matrix of $\g$
Let $\Delta^+ \subset \Delta$ be the system of positive roots corresponding to the
$\alpha$'s. We have the triangular decomposition:
\begin{eqnarray}
   \g = \g_+ \oplus \h \oplus \g_- \qquad \mathrm{where} \qquad 
   \g_\pm = \sum_{\alpha \in \Delta^+} \g_{\pm \alpha} \notag 
\end{eqnarray}

Let $\{\lambda_1,\dots,\lambda_{\ell}\} \subset \h^*$ be the fundamental weights 
and let $\{ H^{(1)}, \dots, H^{(\ell)} \}$ be the fundmental coweights. 
Of course, the fundamental weights are dual to the simple coroots
(i.e. $\lambda_i(H_j)=\delta_{ij}$) and the fundamental coweights are
dual to the simple roots (i.e. $\alpha_i(H^{(j)})=\delta_{ij}$). We should also
note that $H_i = \sum_{i=1}^{\ell} a_{ji}H^{(j)}$.

Let $Q$ and $P$ denote the root lattice and weight lattice of $\g$, respectively: 
\[ Q = \Z \alpha_1 + \cdots + \Z \alpha_{\ell} \quad \mathrm{and} \quad
   P = \Z \lambda_1 + \cdots + \Z \lambda_{\ell} \]
and let $Q^\vee$ and $P^\vee$ denote the coroot lattice and coweight lattice of 
$\g$, respectively:
\[ Q^\vee = \Z H_1 + \cdots + \Z H_{\ell} \quad \mathrm{and} \quad
   P^\vee = \Z H^{(1)} + \cdots + \Z H^{(\ell)}. \]
Define the set of \emph{dominant integral weights} by:
\[ P_+ = \{\lambda \in \h^* | \lambda(H_i) \in \Z_{\ge 0}, \ 1 \le i \le \ell\} \]

For $\lambda \in \h^*$, define the \emph{Verma module} of highest weight $\lambda$ by:
\begin{eqnarray}
V(\lambda) = U(\g) \otimes_{U(\h \oplus \g_+)} \C_\lambda,
\end{eqnarray}
where $\C_\lambda$ is a $1$-dimensional $\h \oplus \g_+$-module given by:
\begin{equation*} \begin{array}{cccc}
h \cdot 1 = \lambda(h) & \qquad \mathrm{for\;all} \qquad & h \in \h  & \qquad \mathrm{and}\\
g \cdot 1 = 0          & \qquad \mathrm{for\;all} \qquad & \ g \in \g_+. & 
\end{array} \end{equation*}
Let $J(\lambda)$ be the maximal proper submodule of $V(\lambda)$. Then 
$L(\lambda) = V(\lambda)/J(\lambda)$ is an irreducible (highest weight) $\g$-module. 
Of course, $L(\lambda)$ is finite dimensional if and only if $\lambda \in P_+$.

Let $\alpha$ be a root, then the map $\sigma_\alpha : \h^* \rightarrow \h^*$ defined by 
\[ \sigma_{\alpha}(\lambda) 
     = \lambda - \frac{2\langle    \alpha,\lambda \rangle }{\|\alpha\|^2}\alpha 
     = \lambda - \frac{2\langle    \alpha,\lambda \rangle }{\langle   \alpha,\alpha\rangle }\alpha \]
is called the \emph{reflection} associated with $\alpha$ (notice that 
$\sigma_\alpha(\alpha)=-\alpha$ and $\sigma_\alpha^2 = \mathrm{Id}_{\h^*}$). 
The group $W$ generated by these reflections is called the \emph{Weyl group} of $\g$. 
Let $\sigma_i = \sigma_{\alpha_i}$ for each simple root $\alpha_i$. The $\sigma_i$'s 
are called \emph{simple reflections}. It is well known that $W$ is generated by simple
reflections. The elements of $W$ are isometries (with respect to the standard 
form) of $\h^*$. Note: We can (and do) transport the action of $W$ on $\h^*$ to an action 
on $\h$ using the standard form.

The (untwisted) affine Lie algebra associated with $\g$ is given by
\begin{eqnarray}
\hat{\g} = \g \otimes \C [t,t^{-1}]\oplus \C c,
\end{eqnarray}
where for $a,b \in \g$, $m,n \in \Z$,
\begin{eqnarray}
 & & [a\otimes t^m, b\otimes t^n] = [a,b] \otimes t^{m+n} + m \langle   a,b\rangle \delta_{m+n,0} c,\\ \notag
 & & \hspace{2cm} \ \ \ \ \  [\hat{\g},c]=0. \notag
\end{eqnarray}

For $a \in \g, n \in \Z$, let $a(n)$ denote the action of $a\otimes t^{n} \in \hat{\g}$ on a $\hat{\g}$-module.
Let $\theta = \sum_{i=1}^\ell a_i\alpha_i$ be the highest long root of $\g$, and choose (non-zero) vectors 
$E_{\theta} \in \g_{\theta}$ and $F_{\theta} \in \g_{-\theta}$ such that $\langle    E_{\theta}, F_{\theta}\rangle =1$. 
Let $H_{\theta} = [ E_{\theta},F_{\theta}]$.  Define 
\begin{equation} \label{chev-gen0}
   e_0 = F_{\theta} \otimes t, \qquad 
   f_0 = E_{\theta} \otimes t^{-1}, \qquad \mathrm{and} \qquad
   h_0 = [e_0,f_0] 
\end{equation}
and for $i = 1 \cdots \ell$ define:
\begin{equation} 
   e_i = E_i \otimes 1, \qquad
   f_i = F_i \otimes 1, \qquad \mathrm{and} \qquad
   h_i = H_i \otimes 1. 
\end{equation}
Then, by \cite{kac-book},
$\{ e_i, f_i, h_i \,|\, 0 \leq i \leq \ell\}$ is a set of Chevalley generators for $\hat{\g}$.

Let $\hat{\g}_{>0} = \g \otimes t\C[t]$ and $\hat{\g}_0 = \g \otimes 1 \oplus \C c$.
Also, let $\hat{\g}_{\ge 0} = \hat{\g}_{>0} \oplus \hat{\g}_0 = \g \otimes \C[t] \oplus \C c$
and fix a scalar $k \in \C$. We can make $L(\lambda)$ into a $\hat{\g}_{\ge 0}$-module 
by extending the action of $\g = \g \otimes 1$ as follows:
\begin{equation*} \begin{array}{cccc}
c \cdot v = kv & \qquad \mathrm{for\;all} \qquad & v \in L(\lambda)  & \qquad \mathrm{and}\\
x \cdot v = 0  & \qquad \mathrm{for\;all} \qquad & \ x \in \g_{>0}. & 
\end{array} \end{equation*}
We now induce up to a $\hat{\g}$-module:
\[ V(k,\lambda) = U(\hat{\g}) \otimes_{U(\hat{\g}_{\ge 0})} L(\lambda). \]
$V(k,\lambda)$ is a \emph{generalized Verma module}.

Unless $k$ is the negative of the dual Coxeter number of $\g$, $V(k,0)$ has the structure of a vertex operator algebra. For such $k$, $V(k,\lambda)$ is a weak $V(k,0)$-module for
all $\lambda \in \h^*$. Moreover, $V(k,\lambda)$ is a (strong) $V(k,0)$-module 
if (and only if) $\lambda \in P_+$.

Let $J(k,\lambda)$ be the maximal proper submodule of $V(k,\lambda)$, and define 
\[ L(k,\lambda) = V(k,\lambda)/J(k,\lambda). \]
If $k$ is not the negative of the dual Coxeter number of $\g$, then $L(k,0)$ has the structure 
of a simple vertex operator algebra, and each $L(k,\lambda)$ is an irreducible weak (unless $\lambda \in P_+$) $V(k,0)$-module. 

If $k$ is a positive integer, then
\begin{thm} \cite{dl}
   $L(k,\lambda)$ is an $L(k,0)$-module if and only if $\langle    \lambda, \theta \rangle  \le k$, 
   where $\theta$ is the highest long root of $\g$.
\end{thm}

\subsection{Li's $\Delta$-Operators}

Now, let Let $H \in P^{\vee}$. Set 
\begin{eqnarray}
\Delta(H,x) = x^{H(0)}\exp \left(\sum_{n\ge 1}\frac{H(n)}{-n}(-x)^{-n}\right)
\end{eqnarray}
(recall the $H(n)$ is the action of $H \otimes t^n$ on a $\hat{g}$-module).

Note that $\Delta(H,x)$ enjoys the following properties:
\begin{eqnarray}
\Delta(H_1 + H_2,x) = \Delta(H_1,x) \Delta(H_2,x),\\
& &\hspace{-5cm} \Delta(0,x) = \mathrm{Id}.
\end{eqnarray}

We fix the notation $(L(k,\lambda)^{(H)}, Y^{(H)}_{L(k,\lambda)}(\cdot,              x)) 
                   = (L(k,\lambda),       Y_{L(k,\lambda)}     (\Delta(H,x) \cdot , x))$.

For $v \in L(k,0)$, set
\begin{equation}
Y_{L(k,\lambda)}^{(H)}(v,x) = Y_{L(k,\lambda)}(\Delta(H,x)v,x) 
                            = \sum_{n \in \Z} v^{(H)}(n) x^{-n-1}. 
\end{equation}

Let us take care of some preliminary calculations involving $\Delta (H,x)$. 
For $g \in {\g}_{\beta}$, $\beta \in \Delta \cup \{0\}$, so that $[h,g] = \beta(h)g$ for all $h \in \h$ 
(Here, ${\g}_0 = \h$), we have:
\begin{eqnarray}
[h(m),g(-1)] & = & [h,g](m-1)\1 + m \langle    h,g \rangle  (m-1)\1 + \delta_{m-1,0}k\1 \\
             & = & \beta (h)g(m-1)\1 + \langle    h,g \rangle  \delta_{m,1}k\1\\
             & = & \begin{cases} \beta (h)g(-1)\1                    & m=0 \\ 
                                 \langle    h,g \rangle  k\delta_{m,1}\1 & m>0 \end{cases}.
\end{eqnarray}
Note that $g(m-1)\1=0$ for $m>0$ by the creation axiom.

We also have (for any $g \in \g$):
\begin{eqnarray*}
  \left( \sum_{m \geq 1} \frac{H(m)}{-m} (-x)^{-m} \right)g = \langle   H,g\rangle  k \1 x^{-1}. 
\end{eqnarray*}
and for $n \ge 2$ 
\begin{equation}
\left(\sum_{m \geq 1}\frac{H(m)}{-m} (-x)^{-m} \right)^n g = 
k\langle   H,g\rangle  x^{-1} \left( \sum_{m \geq 1}\frac{H(m)}{-m}(-x)^{-m} \right)^{n-1} {\bf 1} = 0
\end{equation}
(again using the creation axiom, $H(m)\1=0$ for $m \geq 1$).

Therefore,
\begin{eqnarray}
\Delta(H,x)g & = & x^{H(0)}g + x^{H(0)} \langle    H,g \rangle  k\1 x^{-1}\\
& = & x^{\beta(H)}g + \langle    H,g \rangle  k \1 x^{-1}.
\end{eqnarray}

Applying this to our vertex operator map, we have:
\begin{eqnarray}
Y^{(H)}(g,x) & = & Y(\Delta(H,x)g ,v)\\
& = & x^\beta(H)Y(g,x) + \langle    H,g \rangle  k(\mathrm{Id})x^{-1}.
\end{eqnarray}

Let $k$ be a positive integer and $\lambda \in P_+$ (a dominant integral weight).
In addition, assume that $\langle    \lambda, \theta \rangle  \leq k$ so that
$L(k,\lambda)$ is an (irreducible) $L(k,0)$-module.
It was proved in (\cite{li-affine}, Proposition 2.9) that $(L(k,\lambda)^{(H)}, Y^{(H)}_{L(k,\lambda)})$ carries the structure of an irreducible (weak) $L(k,0)$-module. 
However, since $L(k,0)$ is a regular vertex operator algebra, its weak modules are in
fact (strong) modules. So we have that
\begin{prop} There exists a (unique) $\lambda^{(H)} \in P_+$ such that 
$\langle    \lambda^{(H)}, \theta \rangle  \leq k$ and
$(L(k,\lambda))^{(H)} \cong L(k,\lambda^{(H)})$ as $L(k,0)$-modules.
\end{prop}

Also, Li established that 
\begin{thm} (\cite{li-affine}, Proposition 2.25) For $H \in Q^{\vee}$, $(L(k,\lambda))^{(H)}$ and $L(k,\lambda)$ are isomorphic as $L(k,0)$ modules. 
\end{thm}
That is, in our notation, $\lambda^{(H)} = \lambda$ when $H \in Q^{\vee}$.

Our objective is to see what happens when we allow $H$ to be any element of $P^{\vee}$
(not just $Q^{\vee}$).

Because $\Delta(H'+H'',x) = \Delta(H',x)\Delta(H'',x)$, we have the following:
\begin{eqnarray*}
L(k,\lambda^{(H'+H'')}) & \cong & L(k,\lambda)^{(H' + H'')}\\
                        &     = & (L(k,\lambda)^{(H')})^{(H'')}\\
                        & \cong & L(k,\lambda^{(H')})^{(H'')}\\
                        & \cong & L(k,(\lambda^{(H')})^{(H'')})\\
\end{eqnarray*}
Therefore, $\lambda^{(H'+H'')} = (\lambda^{(H')})^{(H'')}$.

The coweights $\{ H^{(1)},\dots,H^{(\ell)} \}$ form a basis for $P^{\vee}$ and so, by the observation above, all we need to find is action of $\Delta(H^{(i)},x)$. Let us fix the
notation $(L(k,\lambda))^{(i)} = (L(k,\lambda))^{(H^{(i)})}$ and $\lambda^{(i)} = \lambda^{(H^{(i)})}$ for 
$i=1,\dots,\ell$ that is $L(k,\lambda^{(i)}) \cong (L(k,\lambda))^{(i)}=(L(k,\lambda))^{(H^{(i)})}$.

Actually, since the operators associated with coroots act trivially, we only need to
consider one representative for each coset of $P^\vee / Q^\vee$. This pares down the list
of coweights considerably.

We know that the irreducible $L(k,0)$-modules are precisely the irreducible integrable
highest weight modules (i.e. standard modules) for $\hat{g}$ (\cite{dl}).
Now, an irreducible highest weight module is completely determined by its 
highest weight. Therefore, if we can locate a highest weight vector in $(L(k,\lambda))^{(i)}$ and measure its weight, we have determined $\lambda^{(i)}$ and thus the action of $\Delta(H^{(i)},x)$ on $L(k,\lambda)$.

Now, for our Chevalley generators $e_i, f_i, h_i \ 0 \le i \le \ell$, we define ${e_i}^{(j)}$ to be the action of $e_i$ on $L(k,\lambda^{(j)})$, ${f_i}^{(j)}$ to be the action of $f_i$ on $L(k,\lambda^{(j)})$, and ${h_i}^{(j)}$ to be the action of $h_i$ on $L(k,\lambda^{(j)})$. 
Let us calculate these actions:
\begin{eqnarray*}
Y^{(j)}(H_i,x) & = & Y(\Delta(H^{(j)},x)H_i,x)\\
& = & Y(H_i,x) + \frac{2k}{\| \alpha_i \| ^2} \delta_{i,j} x^{-1}
\end{eqnarray*}
and so we have that 
\begin{equation}
{(H_i)}^{(j)}(0) = H_i(0) + \frac{2k}{\| \alpha_i \| ^2} \delta_{i,j}. \label{hi-delaction}
\end{equation}

For $e_i, \ 1 \le i \le \ell$, we have
\begin{eqnarray*}
Y^{(j)}(E_i,x) & = & Y(\Delta(H^{(j)},x)E_i,x)\\
&=& x^{\alpha_i(H^{(j)})}Y(E_i,x)\\
&=& x^{\delta_{i,j}}Y(E_i,x)
\end{eqnarray*}
and so 
\begin{equation}
{e_i}^{(j)} = {E_i}^{(H^{(j)})}(0) = E_i(\delta_{i,j}). \label{ei-delaction}
\end{equation}

Finally, for $e_0$, we have:
\begin{eqnarray*}
Y^{(j)}(F_\theta ,x) & = & Y(\Delta(H^{(j)},x)F_{\theta} ,x)\\
& = & x^{-\theta(H^{(j)})}Y(F_{\theta} ,x) \\
& = & x^{- a_j} Y(F_\theta ,x),
\end{eqnarray*}
(recall that $\theta = \sum_{i=1}^\ell a_i \alpha_i$) and so 
\begin{equation}
{e_0}^{(j)} = {F_\theta}^{(j)}(1) = F_\theta(1-a_j). \label{e0-delaction}
\end{equation}

\section{The $sl_2(\C)$ case}

Before answering our question for all finite dimensional simple Lie algebras let us
consider the simplest case -- that of $\g = sl_2(\C)$.

Let $\displaystyle{E = \begin{pmatrix} 0 & 1 \\ 0 & 0 \end{pmatrix}}$, 
$\displaystyle{F = \begin{pmatrix} 0 & 0 \\ 1 & 0 \end{pmatrix}}$, and
$\displaystyle{H = \begin{pmatrix} 1 & 0 \\ 0 & -1 \end{pmatrix}}$. We know that $E,F,H$ are Chevalley generators of $sl_2(\C)$.  Let $\alpha$ be the fundamental root of $sl_2(\C)$ (that
is, $\alpha(H)=2$), and $\lambda = \frac{1}{2} \alpha$ be the fundamental weight of $sl_2(\C)$.   
Since $\alpha$ is the only positive root, it is the highest long root -- that is 
$\alpha=\theta$.

Then, $H$ is the coroot corresponding to $\alpha$ and $H^{(1)} = \frac{1}{2}H$ is the coweight corresponding to $\lambda$.

In the $sl_2(\C)$ case, we denote $L(k,n\lambda)$ by $L(k,n)$. Let $k$ be a positive integer
and $n \in \Z$ such that $0 \leq n \leq k$.
\begin{thm} $(L(k,n))^{(1)} \cong L(k,k-n)$\\ 
Moreover, if $v$ is a highest weight vector for $L(k,n)$, then $F(0)^n v$ is a highest weight
vector for $(L(k,n))^{(1)}$ (with weight $\Lambda = (k,(k-n)\lambda)$).
\end{thm}
\begin{proof}
Since $-\theta=-\alpha$, we have $F = F_\theta$ and so, recalling (\ref{chev-gen0}), we
have $e_0 = F_\theta \otimes t = F \otimes t$. By (\ref{e0-delaction}), we
have that ${e_0}^{(1)}\cdot w = F(0) \cdot w$. Therefore,
\[ {e_0}^{(1)}\cdot F(0)^nv = F(0) \cdot F(0)^nv = F(0)^{n+1} \cdot v = 0 \]
by the representation theory of $sl_2(\C)$.

Next, we have
$${e_1}^{(1)} \cdot F(0)^nv = E(1)F(0)^nv.$$
Since $L(k,n)$ is a vertex operator algebra module, we can consider conformal weights.
Notice that as operators $\wt F(0) = \wt(F) -0 -1 = 0$ and $\wt E(1) = \wt(E) -1-1 = -1$.
This implies that ${e_1}^{(1)} \cdot F(0)^nv$ has a lower (conformal) weight than $v$. 
However, the highest weight vector occupies the lowest (conformal) weight space of $L(k,n)$. 
Thus ${e_1}^{(1)}F(0)^nv = 0$.
Therefore, $F(0)^nv$ is a highest weight vector for $(L(k,n))^{(1)}$.

Now let us determine the weight of $F(0)^nv$. We will use the fact that $[H,F^n] = -2nF^n$
(in $U(sl_2)$). 
\begin{eqnarray*}
(H)^{(1)} \cdot F(0)^nv & = & H(0)F(0)^nv + kF(0)^nv \\
                        & = & F(0)^nH(0)v + [H(0),F(0)^n]v + kF(0)^nv \\
                        & = & nF(0)^nv - 2nF(0)^nv + kF(0)^nv\\
                        & = & (k-n)F(0)^nv.
\end{eqnarray*}
Hence, $(L(k,n))^{(1)} \cong L(k,k-n)$.
\end{proof}

\begin{rem}
The lowest conformal weight space of $L(k,n)$ is a copy of $L(n\lambda)$ (the
finite dimensional $sl_2$-module with highest weight $n\lambda$). The highest
weight vector of $L(k,n)$ is located in this (lowest) conformal weight space. 

Observe that the new highest weight vector for $(L(k,n))^{(1)}$ is also located
in this copy of $L(n\lambda)$. In fact, we can obtain the new highest weight vector
from the old one by applying the (only) reflection in the Weyl group of $sl_2$.
\end{rem}

\section{The General Case}

Recall that, given an irreducible $L(k,0)$-module $W$ and coroot $H$, $W^{(H)}$ and 
$W$ are isomorphic as $L(k,0)$-modules. This, along with the fact that $W^{(H'+H'')}$ 
is isomorphic to $(W^{(H')})^{(H'')}$, implies that we only need to consider one 
representative from each distinct coset of $P^\vee/Q^\vee$. It is well known that 
$H=0$ (which acts as the identity -- $W^{(0)} = W$), along with the miniscule coweights, 
give us a complete set of coset representatives. We give a list of such coweights below:

\begin{center}
\begin{tabular}{|c|c|c|c|c|c|c|c|} \hline 
Type & $A_{\ell}$ & $B_{\ell}$ & $C_{\ell}$ & $D_{\ell}$ & $E_{6}$ & $E_{7}$ \\ \hline
Coweights & $H^{(1)},\dots,H^{(\ell)}$ & $H^{(1)}$ & $H^{(\ell)}$ &  
$H^{(1)}$,  $H^{(\ell-1)}$,  $H^{(\ell)}$ & $H^{(1)}$,  $H^{(6)}$ & $H^{(7)}$ \\ \hline
\end{tabular}
%%% Give this table an equation number
\begin{equation} 
   \label{validCoweights} \mathrm{Table\;} \ref{validCoweights} \notag 
\end{equation}
\addtocounter{equation}{1}
%%%
\end{center}

Types $E_8$, $F_4$, and $G_2$ have no miniscule coweights, so the action of $\Delta(H,x)$
is always trivial.

Given any coweight $H$ from this list, we wish to determine $\lambda^{(H)}$ such
that $(L(k,\lambda))^{(H)}$ is isomorphic to $L(k,\lambda^{(H)})$. To do this, we need to identify a highest weight vector for $(L(k,\lambda))^{(H)}$ and then 
measure its weight. We will see that in each case our ``new'' highest weight
vector (for $(L(k,\lambda))^{(H)}$) is located in the lowest (conformal) weight
space. This lowest (conformal) weight space is a copy of $L(\lambda)$ -- the
finite dimensional $\g$-module from which $L(k,\lambda)$ is built. In fact,
our ``new'' highest weight vector can be obtained from our ``old'' highest
weight vector by applying the following Weyl group elements
(For each type $X_\ell$ and coweight $H^{(j)}$, define and element $\sigma_{X}^{(j)}$ as follows):

\begin{center}
\begin{tabular}{|l|r|} \hline
    Type & Weyl Group Element \\ \hline
    $A_{\ell}$ &  $1 \leq j \leq \ell$, \quad $\sigma_{A}^{(j)} = (\sigma_1 \sigma_2 \cdots \sigma_\ell)^j$ \\ \hline
    $B_{\ell}$ &  $\sigma_{B}^{(1)} = \sigma_1 \sigma_2  \cdots \sigma_{\ell-1} \sigma_\ell \sigma_{\ell-1} \cdots \sigma_2 \sigma_1$ \\ \hline
    $C_{\ell}$ &  $\sigma_{C}^{(\ell)} = (\sigma_\ell \sigma_{\ell-1} \cdots \sigma_1) (\sigma_\ell \sigma_{\ell-1} \cdots \sigma_2) \cdots (\sigma_\ell \sigma_{\ell-1}) (\sigma_\ell)$ \\ \hline
    $D_\ell$ &  $\sigma_{D}^{(1)} = \sigma_1 \sigma_2 \cdots \sigma_\ell \sigma_{\ell-2} \sigma_{\ell-3} \cdots \sigma_2 \sigma_1$ \\ \hline
    $\ell$ even &  $\sigma_{D}^{(\ell - 1)} = (\sigma_{\ell-1} \sigma_{\ell-2} \sigma_{\ell-3} \cdots \sigma_2 \sigma_1) (\sigma_{\ell} \sigma_{\ell-2} \sigma_{\ell-3} \cdots \sigma_3 \sigma_2) (\sigma_{\ell-1} \sigma_{\ell-2} \sigma_{\ell-3} \cdots \sigma_3) \cdots (\sigma_{\ell-1}) $ \\ \hline
                &  $ \sigma_{D}^{(\ell)} = (\sigma_\ell \sigma_{\ell-2} \sigma_{\ell-3} \cdots \sigma_2 \sigma_1) (\sigma_{\ell-1} \sigma_{\ell-2} \sigma_{\ell-3} \cdots \sigma_3 \sigma_2) (\sigma_\ell \sigma_{\ell-2} \sigma_{\ell-3} \cdots \sigma_3) \cdots (\sigma_\ell) $\\ \hline
    $\ell$ odd &  $\sigma_{D}^{(\ell -1)} =  (\sigma_{\ell-1} \sigma_{\ell-2} \sigma_{\ell-3} \cdots \sigma_2 \sigma_1) (\sigma_{\ell} \sigma_{\ell-2} \sigma_{\ell-3} \cdots \sigma_3 \sigma_2) (\sigma_{\ell-1} \sigma_{\ell-2} \sigma_{\ell-3} \cdots \sigma_4 \sigma_3) \cdots (\sigma_{\ell}) $\\ \hline
               &  $ \sigma_{D}^{(\ell)} = (\sigma_{\ell} \sigma_{\ell-2} \sigma_{\ell-3} \cdots \sigma_2 \sigma_1) (\sigma_{\ell-1} \sigma_{\ell-2} \sigma_{\ell-3} \cdots \sigma_3 \sigma_2) (\sigma_{\ell} \sigma_{\ell-2} \sigma_{\ell-3} \cdots \sigma_4 \sigma_3) \cdots (\sigma_{\ell -1}) $\\ \hline
    $E_6$ &  $\sigma_{E}^{(1)} = \sigma_1 \sigma_3 \sigma_4 \sigma_2\sigma_5\sigma_4\sigma_3\sigma_1\sigma_6\sigma_5\sigma_4\sigma_2\sigma_3\sigma_4\sigma_5\sigma_6$ \\ \hline
    $E_6$ &  $\sigma_{E}^{(6)} = \sigma_6\sigma_5\sigma_4\sigma_2\sigma_3\sigma_4\sigma_5\sigma_6\sigma_1\sigma_3\sigma_4\sigma_2\sigma_5\sigma_4\sigma_3\sigma_1 $ \\ \hline
    $E_7$ &  $\sigma_{E}^{(7)} = \sigma_7\sigma_6\sigma_5\sigma_4\sigma_3\sigma_2\sigma_4\sigma_5\sigma_6\sigma_7\sigma_1\sigma_3\sigma_4\sigma_5\sigma_6\sigma_2\sigma_4\sigma_5\sigma_3\sigma_4\sigma_1\sigma_3\sigma_2\sigma_4\sigma_5\sigma_6\sigma_7  $ \\ \hline
\end{tabular}
%%% Give this table an equation number
\begin{equation}  
   \label{sigmasDefined} \mathrm{Table\;} \ref{sigmasDefined} \notag 
\end{equation}
\addtocounter{equation}{1}
%%%
\end{center}

Let $w$ be a Weyl group element and ${\bf v} \in L(\lambda)_{\mu}$ (a vector of weight $\mu$).
Then $w({\bf v}) \in L(\lambda)_{w(\mu)}$. We begin by calculating the effects of the Weyl group elements defined above on the fundamental weights and simple roots of $\g$. 
To simplify computations, we will use the following conventions:
\begin{equation*} \begin{array}{rclcrcl}
\lambda_0 & = & 0 & \qquad &
\lambda_j & = & \lambda_{j (\mathrm{mod}(l+1))} \\
\alpha_0 & = & -\theta & \qquad &
\alpha_j & = & \alpha_{j (\mathrm{mod}(l+1))}
\end{array} \end{equation*}

We will repeatedly use the fact that $\sigma_i (\lambda_j) = \lambda_j - \delta_{i,j}\alpha_j$.

\subsection{The Weyl group elements' actions}

Let us begin by calculating the action of $\sigma_A^{(1)}$ on the fundamental weights
and then use that to determine the action of our type $A_{\ell}$ Weyl group elements.

For $\lambda_1$, we have: 
$\sigma_1\sigma_2 \cdots \sigma_\ell(\lambda_1) = \sigma_1(\lambda_1) = -\lambda_1+\lambda_2.$

For $\lambda_j$, with $1 < j < \ell$:
\begin{eqnarray*}
 \sigma_1 \sigma_2 \cdots \sigma_\ell (\lambda_j) & = & \sigma_1 \sigma_2 \cdots \sigma_j (\lambda_j)\\
& = & \sigma_1 \sigma_2 \cdots \sigma_{j-1}(\lambda_{j-1} - \lambda_{j} + \lambda_{j+1}) \\
& = & \sigma_1 \sigma_2 \cdots \sigma_{j-2}(\lambda_{j-2} - \lambda_{j-1} + \lambda_{j+1}) \\
& = & \cdots \\
& = & \sigma_1 (\lambda_1 - \lambda_2 + \lambda_{j+1}) \\
& = & -\lambda_1 + \lambda_{j+1}.
\end{eqnarray*}

Finally, for $\lambda_\ell$:
\begin{eqnarray*}
 \sigma_1 \sigma_2 \cdots \sigma_\ell (\lambda_\ell) & = & \sigma_1 \sigma_2 \cdots \sigma_{\ell-1} (\lambda_{\ell-1} - \lambda_\ell)\\
& = &\sigma_1 \sigma_2 \cdots \sigma_{\ell-2} (\lambda_{\ell-2} - \lambda_{\ell-1})\\
& = & \cdots \\
& = & \sigma_1 (\lambda_1 - \lambda_2) \\
& = & -\lambda_1.
\end{eqnarray*}

Adhering to our above convention, we conclude that for $1 \le j \le l$: 
\begin{equation} \label{alambda}
\sigma_1 \sigma_2 \cdots \sigma_\ell (\lambda_j)  = -\lambda_1 + \lambda_{j+1}.
\end{equation}

Therefore, for $1 \leq j < \ell$ (recall that $\lambda_0=0$), 
\begin{eqnarray*}
 \sigma_1 \sigma_2 \cdots \sigma_\ell(\alpha_j) & = & 
 \sigma_1 \sigma_2 \cdots \sigma_\ell(-\lambda_{j-1} + 2\lambda_j - \lambda_{j+1})\\
& = & -(-\lambda_1+\lambda_{j-1+1})+2(-\lambda_1+\lambda_{j+1})-(-\lambda_1+\lambda_{j+1+1})\\
& = & -\lambda_j + 2\lambda_{j+1} - \lambda_{j+2} = \alpha_{j+1}.
\end{eqnarray*}

For $\alpha_\ell$, we have that 
\begin{eqnarray*}
 \sigma_1 \sigma_2 \cdots \sigma_\ell(\alpha_\ell) & = & \sigma_1 \sigma_2 \cdots \sigma_\ell (-\lambda_{\ell-1} + 2\lambda_\ell) \\
& = & -(-\lambda_1+\lambda{\ell-1+1})+2(-\lambda_1+\lambda_{\ell+1}) \\
& = & -\lambda_1-\lambda_{\ell} = -\theta \;\;(= \alpha_0).
\end{eqnarray*}

For $\alpha_0 = -\theta$, we have:
\begin{eqnarray*}
 \sigma_1 \sigma_2 \cdots \sigma_\ell(-\theta) & = & \sigma_1 \sigma_2 \cdots \sigma_\ell( -\lambda_1 - \lambda_\ell)\\
& = & -(-\lambda_1 + \lambda_2)-(-\lambda_1+\lambda_{\ell+1}) \\
& = & 2\lambda_1-\lambda_2 = \alpha_1.
\end{eqnarray*}

So we have found that for all $0 \leq i \leq \ell$, $\sigma_A^{(1)}(\alpha_i) = \alpha_{i+1}$
and therefore
\begin{equation}
 \sigma_A^{(j)}(\alpha_i) = (\sigma_1 \sigma_2 \cdots \sigma_\ell)^j(\alpha_i) = \alpha_{i+j}. \end{equation}

The calculations for types $B_\ell$, $C_\ell$, and $D_\ell$ are similar and can be 
found in the Appendix. Calculations for types $E_6$ and $E_7$ were done with the help of a Maple worksheet written and updated by the first author which was orginally written for \cite{cms}. This worksheet is available at: 
\verb|http://dimax.rutgers.edu/~sadowski/LieAlgebraCalculations/index.html|

The Weyl group elements' actions can be summed up as follows (recall that $\alpha_0=-\theta$):
\begin{equation*} \begin{array}{rcl} \vspace{0.1in}
A_{\ell}: & \quad & \sigma_A^{(j)}(\alpha_j) = \alpha_{j+i\;(\mathrm{mod}\;\ell+1)} 
                  \;\;(\mathrm{for}\;0 \leq j \leq \ell 
                  \;\mathrm{and}\;\; 1 \leq i \leq \ell) \\ \vspace{0.1in}
B_{\ell}: & \quad & \sigma_B^{(1)}(\alpha_0) = \alpha_1,  
                    \sigma_B^{(1)}(\alpha_1) = \alpha_0, \;\mathrm{and}\;\;
                    \sigma_B^{(1)}(\alpha_j) = \alpha_j 
                  \;\;(\mathrm{for}\;1 < j \leq \ell) \\ \vspace{0.1in}
C_{\ell}: & \quad & \sigma_C^{(\ell)}(\alpha_0) = \alpha_\ell,
                    \sigma_C^{(\ell)}(\alpha_j) = \alpha_{\ell-i} 
                  \;\;(\mathrm{for}\;1 \leq j < \ell), \;\;\mathrm{and}\;\;
                    \sigma_C^{(\ell)}(\alpha_\ell) = \alpha_0 \\ \vspace{0.05in}
D_{\ell}: & \quad & \sigma_D^{(1)}(\alpha_0) = \alpha_1,
                    \sigma_D^{(1)}(\alpha_1) = \alpha_0,
                    \sigma_D^{(1)}(\alpha_j) = \alpha_j 
                  \;\;(\mathrm{for}\;1 < j < \ell-1),\\ \vspace{0.1in} 
          &       & \sigma_D^{(1)}(\alpha_{\ell-1}) = \alpha_\ell, \;\;\mathrm{and}\;\;
                    \sigma_D^{(1)}(\alpha_{\ell}) = \alpha_{\ell-1} \\ \vspace{0.05in} 
          & \ell\;\mathrm{odd} & \sigma_D^{(\ell-1)}(\alpha_0) = \alpha_{\ell-1},
                    \sigma_D^{(\ell-1)}(\alpha_1) = \alpha_{\ell}, \;\;\mathrm{and}\;\;
                    \sigma_D^{(\ell-1)}(\alpha_j) = \alpha_{\ell-j} 
                  \;\;(\mathrm{for}\;1 < j \leq \ell) \\ \vspace{0.05in} 
          &       & \sigma_D^{(\ell)}(\alpha_0) = \alpha_{\ell},
                    \sigma_D^{(\ell)}(\alpha_j) = \alpha_{\ell-j} 
                  \;\;(\mathrm{for}\;1 \leq j < \ell-1), \\ \vspace{0.1in}
          &       & \sigma_D^{(\ell)}(\alpha_{\ell-1}) = \alpha_0, \;\;\mathrm{and}\;\;
                    \sigma_D^{(\ell)}(\alpha_{\ell}) = \alpha_1 \\ \vspace{0.05in} 
          & \ell\;\mathrm{even} & \sigma_D^{(\ell-1)}(\alpha_0) = \alpha_{\ell-1},
                    \sigma_D^{(\ell-1)}(\alpha_1) = \alpha_{\ell},
                    \sigma_D^{(\ell-1)}(\alpha_j) = \alpha_{\ell-j} 
                  \;\;(\mathrm{for}\;1 < j < \ell-1), \\ \vspace{0.1in} 
          &       & \sigma_D^{(\ell-1)}(\alpha_{\ell-1}) = \alpha_0, \;\;\mathrm{and}\;\;
                    \sigma_D^{(\ell-1)}(\alpha_{\ell}) = \alpha_1 \\ \vspace{0.1in}  
          &       & \sigma_D^{(\ell)}(\alpha_j) = \alpha_{\ell-j} 
                  \;\;(\mathrm{for}\;0 \leq j \leq \ell) \\ \vspace{0.05in}
E_6:      & \quad & \sigma_E^{(1)}(\alpha_0) = \alpha_1,
                    \sigma_E^{(1)}(\alpha_1) = \alpha_6,
                    \sigma_E^{(1)}(\alpha_2) = \alpha_3,
                    \sigma_E^{(1)}(\alpha_3) = \alpha_5,
                    \sigma_E^{(1)}(\alpha_4) = \alpha_4, \\ \vspace{0.1in}
          &       & \sigma_E^{(1)}(\alpha_5) = \alpha_2, \;\;\mathrm{and}\;\;
                    \sigma_E^{(1)}(\alpha_6) = \alpha_0 \\ \vspace{0.05in}
          &       & \sigma_E^{(6)}(\alpha_0) = \alpha_6,
                    \sigma_E^{(6)}(\alpha_1) = \alpha_0,
                    \sigma_E^{(6)}(\alpha_2) = \alpha_5,
                    \sigma_E^{(6)}(\alpha_3) = \alpha_2,
                    \sigma_E^{(6)}(\alpha_4) = \alpha_4, \\ \vspace{0.1in}
          &       & \sigma_E^{(6)}(\alpha_5) = \alpha_3, \;\;\mathrm{and}\;\;
                    \sigma_E^{(6)}(\alpha_6) = \alpha_1 \\ \vspace{0.05in}
E_7:      & \quad & \sigma_E^{(7)}(\alpha_0) = \alpha_7,
                    \sigma_E^{(7)}(\alpha_1) = \alpha_6,
                    \sigma_E^{(7)}(\alpha_2) = \alpha_2,
                    \sigma_E^{(7)}(\alpha_3) = \alpha_5,
                    \sigma_E^{(7)}(\alpha_4) = \alpha_4, \\ \vspace{0.05in}
          &       & \sigma_E^{(7)}(\alpha_5) = \alpha_3,
                    \sigma_E^{(7)}(\alpha_6) = \alpha_1, \;\;\mathrm{and}\;\;
                    \sigma_E^{(7)}(\alpha_7) = \alpha_0 \\ \vspace{0.05in}
\end{array} \end{equation*}

\vspace{-0.3in}

%%% Give this table an equation number
\begin{equation}  
   \label{sigmaAction} \mathrm{Table\;} \ref{sigmaAction} \notag 
\end{equation}
\addtocounter{equation}{1}
%%%

\subsection{The Main Theorems}

Now we can apply our Weyl group calculations to determine the highest weight vector
of $(L(k,\lambda))^{(i)}$ (which is simultaneously an $L(k,0)$ and $\hat{\g}$ module). 
Suppose our underlying finite dimensional simple Lie algebra $\g$ is of type $X_\ell$.
Recall that $\theta=\sum_j a_j\alpha_j$ is the highest long root of $\g$, $k$ is a positive integer, $\lambda \in P_+$ with $\langle    \lambda, \theta \rangle  \leq k$, and $H^{(i)}$ is one of the coweights appearing in table \ref{validCoweights}.
Also, note that for each valid choice of index $i$, we have $a_i=1$.

The irreducible $\hat{\g}$-module $L(k,\lambda)$ was built up from the irreducible
$\g$-module $L(\lambda)$. In fact, the lowest (conformal) weight space of $L(k,\lambda)$ 
is merely a copy of $L(\lambda)$. In what follows, let us identify $L(\lambda)$ with this
subspace of $L(k,\lambda)$.

Let ${\bf v}$ be a highest weight vector for $L(k,\lambda)$. Such a vector is also homogeneous
vector of lowest conformal weight in $L(k,\lambda)$. So we have that 
${\bf v} \in L(\lambda) \subset L(k,\lambda)$. 
\begin{thm}\label{main-hwvec}
$\sigma_X^{(i)}({\bf v})$ is a highest weight vector for $(L(k,\lambda))^{(i)}$.
\end{thm}
\begin{proof}
Let ${\bf w} = \sigma_X^{(i)}({\bf v})$ and $\mu = \sigma_X^{(i)}(\lambda)$.

Since $\sigma_X^{(i)}$ is an invertible map, ${\bf w} \not= {\bf 0}$. The weight of ${\bf v}$
(as an element of $L(\lambda)$) is $\lambda$, so the weight of ${\bf w}$ is $\mu$. To establish that ${\bf w}$ is a highest weight vector for $(L(k,\lambda))^{(i)}$, we need to show 
that $(e_j)^{(i)} \cdot {\bf w} = {\bf 0}$ for $j=0,\dots,\ell$.

By (\ref{ei-delaction}), we have
$(e_j)^{(i)} \cdot {\bf w} = E_j(\delta_{ij})({\bf w})$
for $1 \leq j \leq \ell$.

When $j=i$, $(e_i)^{(i)} \cdot {\bf w} = E_i(1)({\bf w})$. Now recall
that the operator $E_i(1)$ lowers (conformal) weights by $-1$. Since ${\bf w}$ is already
a lowest (conformal) weight vector, $(e_i)^{(i)} \cdot {\bf w} = {\bf 0}$.  

Next, when $j\not=i$, $(e_j)^{(i)} \cdot {\bf w} = E_j(0)({\bf w})$ which is
a vector of weight $\mu+\alpha_j$ (in the $\g$-module $L(\lambda)$).

By (\ref{e0-delaction}), we have $(e_0)^{(i)} \cdot {\bf w} = F_\theta(1-a_i)({\bf w})$
$=F_\theta(0)({\bf w})$ which has weight $\mu-\theta = \mu+\alpha_0$.

Therefore, we have reduced our problem to establishing that $\mu+\alpha_j$ for $j\not=i$,
$0 \leq j \leq \ell$ are not weights of the $\g$-module $L(\lambda)$.

Since Weyl group elements permute the weights of $L(\lambda)$, if $\mu+\alpha_j$
is a weight, then $(\sigma_X^{(i)})^{-1}(\mu+\alpha_j)=\lambda+(\sigma_X^{(i)})^{-1}(\alpha_j)$ 
must be a weight as well. However, we can see by inspecting table \ref{sigmaAction}
that $\sigma_X^{(i)}$ permutes the set of simple roots along with $\alpha_0=-\theta$. 
Notice that each case $\sigma_X^{(i)}(\alpha_0)=\alpha_i$. Therefore, $(\sigma_X^{(i)})^{-1}$
maps each $\alpha_j$ where $j\not=i$, $0 \leq j \leq \ell$ to some $\alpha_k$ 
where $1 \leq k \leq \ell$. But $\lambda$ is a highest weight vector for $L(\lambda)$, 
therefore $\lambda+\alpha_k$ ($1\leq k \leq \ell$) is not a weight. This implies that 
$\mu+\alpha_j$ (for $j\not=i$, $0 \leq j \leq \ell$) is not a weight. 
Therefore, ${\bf w}$ is annihilated by the action of each $(e_j)^{(i)}$ (for $0\leq j \leq \ell$) and thus is a highest weight vector for $(L(k,\lambda))^{(i)}$.
\end{proof}

\begin{rem} Elements of the Weyl group are linear transformations, so they always send the
linear functional $\lambda=0$ to itself. This implies that a highest weight vector for 
$L(k,0)$ (i.e. a non-zero scalar multiple of the vaccuum vector) is still a highest 
weight vector for $(L(k,0))^{(i)}$. This special case was discussed in \cite{li-affine}.
\end{rem}

\begin{thm}\label{main-measurement}
Recall $(L(k,\lambda))^{(i)} \cong L(k,\lambda^{(i)})$. 
Let $\lambda = \sum_{j=1}^\ell m_j\lambda_j$. Then we have the following:
\begin{center}
\begin{tabular}{ccrcl}
Type     & & \\ 
$A_\ell$ & & $\lambda^{(i)}$ & $=$ & 
    $\sum_{j=1}^\ell m_j\lambda_{j+i} + (k-\langle    \lambda,\theta \rangle )\lambda_i$ \\
$B_\ell$ & & $\lambda^{(1)}$ & $=$ & 
    $(k-\langle   \lambda,\theta\rangle )\lambda_1 + \sum_{j=2}^\ell m_j\lambda_j$ \\
$C_\ell$ & & $\lambda^{(\ell)}$ & $=$ & 
    $\sum_{j=1}^\ell m_j\lambda_{\ell-j} - (k-\langle   \lambda,\theta\rangle )\lambda_\ell$ \\
$D_\ell$ & & $\lambda^{(1)}$ & $=$ & 
    $(k-\langle    \lambda,\theta \rangle )\lambda_1 + \sum_{j=2}^{\ell-2}m_j\lambda_j + m_{\ell-1}\lambda_\ell + m_\ell\lambda_{\ell-1}$ \\ 
 & ($\ell$ odd) & $\lambda^{(\ell-1)}$ & $=$ & 
    $m_{\ell-1}\lambda_1 + \sum_{j=2}^{\ell-2} m_j\lambda_{\ell-j} + (k - \langle   \lambda,\theta\rangle )\lambda_{\ell-1} + m_1\lambda_\ell$ \\
 & ($\ell$ even) & $\lambda^{(\ell-1)}$ & $=$ & 
    $m_\ell\lambda_1 + \sum_{j=2}^{\ell-2} m_j\lambda_{\ell-j} + (k - \langle   \lambda,\theta\rangle )\lambda_{\ell-1} + m_1\lambda_\ell$ \\
 & ($\ell$ odd) & $\lambda^{(\ell)}$ & $=$ & 
    $m_\ell\lambda_1 + \sum_{j=2}^{\ell-2} m_j\lambda_{\ell-j} + m_1\lambda_{\ell-1} + (k - \langle   \lambda,\theta\rangle )\lambda_\ell$ \\
 & ($\ell$ even) & $\lambda^{(\ell)}$ & $=$ & 
    $m_{\ell-1}\lambda_1 + \sum_{j=2}^{\ell-2} m_j\lambda_{\ell-j} + m_1\lambda_{\ell-1} + (k - \langle   \lambda,\theta\rangle )\lambda_\ell$ \\
$E_6$ & & $\lambda^{(1)}$ & $=$ &
    $(k-\langle   \lambda,\theta\rangle )\lambda_1 + m_5\lambda_2 + m_2\lambda_3 + m_4\lambda_4 + m_3\lambda_5 + m_1\lambda_6$ \\
      & & $\lambda^{(6)}$ & $=$ & 
    $m_6\lambda_1 +  m_3\lambda_2 + m_5\lambda_3 + m_4\lambda_4 + m_2\lambda_5 + (k - \langle   \lambda,\theta\rangle )\lambda_6$\\
$E_7$ & & $\lambda^{(7)}$ & $=$ &
    $m_1\lambda_6 + m_2\lambda_2 + m_3\lambda_5 + m_4\lambda_4 + m_5\lambda_3 + m_6\lambda_1 + (k - \langle   \lambda,\theta\rangle )\lambda_7$
\end{tabular}
\end{center}
\end{thm}
\begin{proof}
Let $\g$ be of type $X_\ell$ and let $H^{(i)}$ be a coweight from table \ref{validCoweights}.
Consider a vector ${\bf u} \in L(\lambda)_\beta \subset L(\lambda) \subset L(k,\lambda)$ 
(i.e. the $\h$-weight of ${\bf u}$ is $\beta$). Consider ${\bf u}$ as a vector in 
$(L(k,\lambda))^{(i)}$ (which is the same vector space as $L(k,\lambda)$). We have a new associated $\hat{\g}$-action and thus a new $\h$-action. In particular, by 
\ref{hi-delaction}), $(H_j)^{(i)}(0) = H_j(0) + \frac{2k}{\|\alpha_i\|^2} \delta_{i,j}$. 
If we restrict $i$ to the indices appearing in table \ref{validCoweights}, we see that
in each case $\|\alpha_i\|^2=2$. Thus $(H_j)^{(i)}(0) = H_j(0)+k\delta_{i,j}$.
Therefore, when ${\bf u}$ is thought of as a weight vector under the new 
$\h$-action, the (new) $\h$-weight of ${\bf u}$ is $\beta+k\lambda_i$.

Let ${\bf v}$ be a highest weight vector for $L(k,\lambda)$. Theorem \ref{main-hwvec} 
states that ${\bf w} = \sigma_X^{(i)}({\bf v})$ is a highest weight vector for
$(L(k,\lambda))^{(i)}$. As an element of the $\g$-module 
$L(\lambda)$, ${\bf w}$ has $\h$-weight $\mu = \sigma_X^{(i)}(\lambda)$. 
Therefore, the $\h$-weight of ${\bf w}$ in terms of the new $\h$-action is 
$\lambda^{(i)} = \mu+k\lambda_i$. 

For $X_\ell=A_\ell$, using (\ref{alambda}), we have
$\mu = \sum_{j=1}^{\ell} m_j(\lambda_{j+i}-\lambda_i)$
$= \sum_{j=1}^{\ell} m_j\lambda_{j+i}-(\sum_{j=1}^{\ell} m_j)\lambda_i)$
$= \sum_{j=1}^{\ell} m_j\lambda_{j+i} - \langle    \lambda,\theta \rangle  \lambda_i$.
Therefore, $\lambda^{(i)} = \sum_{j=1}^{\ell} m_j\lambda_{j+i}-\langle    \lambda,\theta \rangle \lambda_i + k\lambda_i$.

For type $B_\ell$, by (\ref{bactionj}) and (\ref{bactionell}), we have
$\mu = \sum_{j=1}^{\ell-1} m_j(-2\lambda_1+\lambda_j)+m_\ell(-\lambda_1+\lambda_\ell)$
$ = \sum_{j=1}^\ell m_j\lambda_j -2(\sum_{j=1}^{\ell-1}m_j)\lambda_1 +m_\ell\lambda_1$
$ = \sum_{j=2}^\ell m_j\lambda_j - \langle   \lambda,\theta\rangle \lambda_1$.
Therefore, $\lambda^{(1)} = \sum_{j=2}^\ell m_j\lambda_j - \langle   \lambda,\theta\rangle  \lambda_1 + k\lambda_1$.

For type $C_\ell$, by (\ref{caction}), we have 
$\mu = \sum_{j=1}^\ell m_j (\lambda_{\ell-j} - \lambda_\ell)$ 
$ = \sum_{j=1}^\ell m_j \lambda_{\ell-j} - (\sum_{j=1}^\ell m_j)\lambda_\ell$ 
$ = \sum_{j=1}^\ell m_j\lambda_{\ell-j} - \langle   \lambda,\theta\rangle \lambda_\ell$.
Therefore, $\lambda^{(\ell)} = \sum_{j=1}^\ell m_j\lambda_{\ell-j} - \langle   \lambda,\theta\rangle \lambda_\ell + k\lambda_\ell$.

For type $D_\ell$ and $i=1$, by (\ref{daction1j}), (\ref{daction1ell-1}), and 
(\ref{daction1ell}), we have 
$\mu = \sum_{j=1}^{\ell-2} m_j(-2\lambda_1+\lambda_j)+m_{\ell-1}(-\lambda_1+\lambda_\ell) + m_\ell(-\lambda_1+\lambda_{\ell-1})$
$ = \sum_{j=1}^{\ell-2} m_j(-2\lambda_1+\lambda_j)+m_{\ell-1}(-\lambda_1+\lambda_\ell) + m_\ell(-\lambda_1+\lambda_{\ell-1})$
$ = -\langle    \lambda,\theta \rangle  \lambda_1 + \sum_{j=2}^{\ell-2}m_j\lambda_j + m_{\ell-1}\lambda_\ell + m_\ell\lambda_{\ell-1}$. 
Therefore, $\lambda^{(1)} = -\langle    \lambda,\theta \rangle  \lambda_1 + \sum_{j=2}^{\ell-2}m_j\lambda_j + m_{\ell-1}\lambda_\ell + m_\ell\lambda_{\ell-1}+k\lambda_1$. 

Now, consider type $D_\ell$ when $\ell$ is odd and $i=\ell-1$. By 
(\ref{dactionell-11}), (\ref{dactionell-1j}), (\ref{dactionell-1ell-1}), and
(\ref{dactionell-1ell}), we have 
$\mu = m_1(-\lambda_{\ell-1}+\lambda_\ell) + \sum_{j=2}^{\ell-2} m_j(\lambda_{\ell-j}-2\lambda_{\ell-1} + m_{\ell-1}(\lambda_1 - \lambda_{\ell-1}) + m_\ell(-\lambda_{\ell-1}$
$ = m_{\ell-1}\lambda_1 + \sum_{j=2}^{\ell-2} m_j\lambda_{\ell-j} - (m_1 + 2\sum_{j=2}^{\ell-2}m_j +m_{\ell-1} + m_\ell)\lambda_{\ell-1} + m_1\lambda_\ell$
$ = m_{\ell-1}\lambda_1 + \sum_{j=2}^{\ell-2} m_j\lambda_{\ell-j} - \langle   \lambda,\theta\rangle \lambda_{\ell-1} + m_1\lambda_\ell$.
Therefore, $\lambda^{(\ell-1)} = m_{\ell-1}\lambda_1 + \sum_{j=2}^{\ell-2} m_j\lambda_{\ell-j} - \langle   \lambda,\theta\rangle \lambda_{\ell-1} + m_1\lambda_\ell + k\lambda_{\ell-1}$.
Considering a Dynkin diagram symmetry (interchanging the roles of nodes $\ell-1$ and $\ell$),
we also have $\lambda^{(\ell)} = m_\ell\lambda_1 + \sum_{j=2}^{\ell-2} m_j\lambda_{\ell-j} 
+ m_1\lambda_{\ell-1} - \langle   \lambda,\theta\rangle \lambda_\ell + k\lambda_\ell$.

When $\ell$ is even and $i=\ell-1$, by (\ref{dactioneven}), we have
$\mu = m_1(-\lambda_{\ell-1}+\lambda_\ell) + \sum_{j=2}^{\ell-2} m_j(\lambda_{\ell-j}-2\lambda_{\ell-1} + m_{\ell-1}(-\lambda_{\ell-1}) + m_\ell(\lambda_1-\lambda_{\ell-1}$
$ = m_\ell\lambda_1 + \sum_{j=2}^{\ell-2} m_j\lambda_{\ell-j} - (m_1 + 2\sum_{j=2}^{\ell-2}m_j +m_{\ell-1} + m_\ell)\lambda_{\ell-1} + m_1\lambda_\ell$
$ = m_\ell\lambda_1 + \sum_{j=2}^{\ell-2} m_j\lambda_{\ell-j} - \langle   \lambda,\theta\rangle \lambda_{\ell-1} + m_1\lambda_\ell$.
Therefore, $\lambda^{(\ell-1)} = m_\ell\lambda_1 + \sum_{j=2}^{\ell-2} m_j\lambda_{\ell-j} - \langle   \lambda,\theta\rangle \lambda_{\ell-1} + m_1\lambda_\ell + k\lambda_{\ell-1}$.
Again, considering a Dynkin diagram symmetry (interchanging the roles of nodes $\ell-1$ 
and $\ell$), we also have $\lambda^{(\ell)} = m_{\ell-1}\lambda_1 + \sum_{j=2}^{\ell-2} m_j\lambda_{\ell-j} + m_1\lambda_{\ell-1} - \langle   \lambda,\theta\rangle \lambda_\ell + k\lambda_\ell$.

Finally, let us consider types $E_6$ and $E_7$. For type $E_6$ and $i=1$, by 
(\ref{eaction1}), we have
$\mu = m_1(-\lambda_1 + \lambda_6) + m_2(-2\lambda_1 + \lambda_3) + m_3(-2\lambda_1 + \lambda_5) + m_4(-3\lambda_1 + \lambda_4) + m_5(-2\lambda_1 + \lambda_2) + m_6(-\lambda_1)$
$ = -(m_1 + 2m_2 + 2m_3 + 3m_4 + 2m_5 + m_6)\lambda_1 + m_5\lambda_2 + m_2\lambda_3
+ m_4\lambda_4 + m_3\lambda_5 + m_1\lambda_6$
$ = -\langle   \lambda,\theta\rangle \lambda_1 + m_5\lambda_2 + m_2\lambda_3
+ m_4\lambda_4 + m_3\lambda_5 + m_1\lambda_6$.
Therefore, $\lambda^{(1)} = -\langle   \lambda,\theta\rangle \lambda_1 + m_5\lambda_2 + m_2\lambda_3 + m_4\lambda_4 + m_3\lambda_5 + m_1\lambda_6 + k\lambda_1$.
Now we can use the symmetry of the Dynkin diagram of $E_6$ (interchanging nodes $1$ and $6$
and also interchanging nodes $3$ and $5$). Therefore, for type $E_6$ with $i=6$, we have
$\lambda^{(6)} = m_6\lambda_1 +  m_3\lambda_2 + m_5\lambda_3 + m_4\lambda_4 + m_2\lambda_5 - \langle   \lambda,\theta\rangle \lambda_6 + k\lambda_6$.

For type $E_7$ and $i=7$, by (\ref{eaction7}), we have
$\mu = m_1(-2\lambda_7 + \lambda_6) + m_2(-2\lambda_7 + \lambda_2) + m_3(-3\lambda_7 + \lambda_5) + m_4(-4\lambda_7 + \lambda_4) + m_5(-3\lambda_7 + \lambda_3) + m_6(-2\lambda_7 + \lambda_1) + m_7(-\lambda_7)$
$ = m_1\lambda_6 + m_2\lambda_2 + m_3\lambda_5 + m_4\lambda_4 + m_5\lambda_3 + m_6\lambda_1  
- (2m_1 + 2m_2 + 3m_3 + 4m_4 + 3m_5 + 2m_6 + m_7)\lambda_7$
$ = m_1\lambda_6 + m_2\lambda_2 + m_3\lambda_5 + m_4\lambda_4 + m_5\lambda_3 + m_6\lambda_1 
- \langle   \lambda,\theta\rangle \lambda_7$.
Therefore, $\lambda^{(7)} = m_1\lambda_6 + m_2\lambda_2 + m_3\lambda_5 + m_4\lambda_4 + m_5\lambda_3 + m_6\lambda_1 - \langle   \lambda,\theta\rangle \lambda_7 + k\lambda_7$.
\end{proof}

\section{Appendix - Weyl group calculations} 

All of these calculations repeatedly use the fact that
$\sigma_i (\lambda_j) = \lambda_j - \delta_{i,j}\alpha_j$, so in particular,
$\sigma_{j-1}\sigma_{j-2}\cdots\sigma_2\sigma_1(\lambda_j)=\lambda_j$. 
In addition, if $C = (a_{ij})$ is the Cartan matrix of our simple Lie algebra $\g$,
then $\alpha_i = \sum_j a_{ij}\lambda_j$. Also, recall our conventions that 
$\lambda_0=0$, $\alpha_0=-\theta$ (the negative of the highest long root of $\g$),
and $\lambda_j = \lambda_{j \;(\mathrm{mod}\; \ell+1)}$ for all $j \in \Z$.

\subsection{Type $B_\ell$}

Looking at the Cartan matrix of type $B_\ell$, we see that
$\alpha_i = -\lambda_{i-1} + 2\lambda_i - \lambda_{i+1}$ for $i\not=\ell-1$
and $\alpha_{\ell-1}=-\lambda_{\ell-2}+2\lambda_{\ell-1}-2\lambda_\ell$. 
Also, recall that $\sigma_{B}^{(1)} = \sigma_1 \sigma_2 \cdots \sigma_{\ell-1} \sigma_\ell \sigma_{\ell-1} \cdots \sigma_2 \sigma_1$

For $\lambda_j$, $1 \le j \le \ell-1$, we have:
\begin{eqnarray} \notag
\sigma_{B}^{(1)}(\lambda_j)
& = & \sigma_1 \sigma_2  \cdots \sigma_{\ell-1} \sigma_\ell \sigma_{\ell-1} \cdots \sigma_2 \sigma_1(\lambda_j) \; = \;
\sigma_ 1 \sigma_2  \cdots \sigma_{\ell-1} \sigma_\ell \sigma_{\ell-1} \cdots \sigma_{j+1} \sigma_j(\lambda_j) \\ \notag
& = & \sigma_1 \sigma_2  \cdots \sigma_{\ell-1} \sigma_\ell \sigma_{\ell-1} \cdots \sigma_{j+2} \sigma_{j+1} (\lambda_{j-1} - \lambda_j + \lambda_{j+1}) \\ \notag
& = & \sigma_1 \sigma_2  \cdots \sigma_{\ell-1} \sigma_\ell \sigma_{\ell-1} \cdots \sigma_{j+3} \sigma_{j+2} (\lambda_{j-1} - \lambda_{j+1} + \lambda_{j+2}) \; = \; \cdots \\ \notag
& = & \sigma_1 \sigma_2 \cdots \sigma_\ell \sigma_{\ell-1} (\lambda_{j-1} - \lambda_{\ell-2} + \lambda_{\ell-1}) \; = \;
\sigma_1 \sigma_2 \cdots \sigma_\ell \sigma_{\ell} (\lambda_{j-1} -\lambda_{\ell-1} + 2\lambda_\ell)\\ \notag
& = & \sigma_1 \sigma_2 \cdots \sigma_{\ell-1} (\lambda_{j-1} + \lambda_{\ell-1} -2\lambda_\ell) \; = \;
\sigma_1 \sigma_2 \cdots \sigma_{\ell-2} (\lambda_{j-1} + \lambda_{\ell-2} - \lambda_{\ell-1})\\ \notag
& = & \sigma_1 \sigma_2 \cdots \sigma_{\ell-3} (\lambda_{j-1} + \lambda_{\ell-3} -\lambda_{\ell-2}) \; = \; \cdots \\ \notag
& = & \sigma_1 \sigma_2 \cdots \sigma_{j}( \lambda_{j-1} +\lambda_j - \lambda_{j+1}) 
\; = \; \sigma_1 \sigma_2 \cdots \sigma_{j-1}( 2\lambda_{j-1} -\lambda_j) \\ \notag
& = & \sigma_1 \sigma_2 \cdots \sigma_{j-2}(2\lambda_{j-2} - 2\lambda_{j-1} + \lambda_j)
\; = \; \cdots \\
& = & \sigma_1 (2\lambda_1 -2\lambda_2 + \lambda_j) 
\; = \; -2\lambda_1 + \lambda_j. \label{bactionj}
\end{eqnarray}

For $\lambda_\ell$, we have:
\begin{eqnarray} \notag
\sigma_{B}^{(1)}(\lambda_\ell) & = &
 \sigma_1 \sigma_2  \cdots \sigma_{\ell-1} \sigma_\ell \sigma_{\ell-1} \cdots \sigma_2 \sigma_1(\lambda_\ell) 
\; = \; \sigma_1 \sigma_2  \cdots \sigma_{\ell-1} \sigma_\ell (\lambda_\ell) \\ \notag
& = & \sigma_1 \sigma_2  \cdots \sigma_{\ell-1} (\lambda_{\ell-1} - \lambda_\ell)
\; = \; \sigma_1 \sigma_2  \cdots \sigma_{\ell-2} (\lambda_{\ell-2} -\lambda_{\ell-1} + \lambda_\ell) \\ \notag
& = & \sigma_1 \sigma_2 \cdots \sigma_{\ell-3}(\lambda_{\ell-3} - \lambda_{\ell-2} + \lambda_\ell) \; = \; \cdots \\
& = & \sigma_1(\lambda_1 - \lambda_2 + \lambda_\ell)
\; = \; -\lambda_1 + \lambda_\ell. \label{bactionell}
\end{eqnarray}

Applying these results to the fundamental roots and highest long root, we have:
\begin{eqnarray*}
 \sigma_{B}^{(1)}(\alpha_0) & = & \sigma_{B}^{(1)}(-\theta) = 
 \sigma_{B}^{(1)}(-\lambda_2) = 2\lambda_1-\lambda_2 = \alpha_1 \\
 \sigma_{B}^{(1)}(\alpha_1) & = & \sigma_{B}^{(1)}(2\lambda_j-\lambda_{j+1})
 = -2\lambda_1+2\lambda_1-\lambda_2 = -\theta = \alpha_0 \\
 \sigma_{B}^{(1)}(\alpha_j) & = & \alpha_j \qquad \mathrm{for} \; 1 < j \le \ell
\end{eqnarray*}

\subsection{Type $C_\ell$}

Looking at the Cartan matrix of type $C_\ell$, we see that
$\alpha_i = -\lambda_{i-1} + 2\lambda_i - \lambda_{i+1}$ for $1 \le i < \ell$
and $\alpha_\ell=-2\lambda_{\ell-1}+2\lambda_\ell$. 
Also, recall that $\sigma_{C}^{(\ell)} = (\sigma_\ell \cdots \sigma_2 \sigma_1)
(\sigma_\ell \cdots \sigma_2)\cdots(\sigma_\ell\sigma_{\ell-1})(\sigma_\ell)$.

\begin{eqnarray*}
   \sigma_\ell \sigma_{\ell-1} \cdots \sigma_1(\lambda_j) 
& = & \sigma_\ell\sigma_{\ell-1}\cdots\sigma_j(\lambda_j) 
\; = \; \sigma_\ell\sigma_{\ell-1}\cdots\sigma_{j+1}(\lambda_{j-1}-\lambda_j+\lambda_{j+1}) \\
& = & \sigma_\ell\sigma_{\ell-1}\cdots\sigma_{j+2}(\lambda_{j-1}-\lambda_{j+1}+\lambda_{j+2}) 
\; = \; \cdots \\
& = & \sigma_\ell(\lambda_{j-1}-\lambda_{\ell-1}+\lambda_\ell) 
\; = \; \lambda_{j-1}+\lambda_{\ell-1}-\lambda_\ell
\end{eqnarray*}
where $1 \le j < \ell$. Also, 
$\sigma_\ell \sigma_{\ell-1} \cdots \sigma_1(\lambda_\ell) = \sigma(\lambda_\ell)
= 2\lambda_{\ell-1}-\lambda_\ell$, so the above formula works for all $j$. Applying
this result multiple times, we have (for $1 \leq j < \ell$):
\begin{eqnarray} \notag
\sigma_{C}^{(\ell)}(\lambda_j) & = & (\sigma_\ell \sigma_{\ell-1} \cdots \sigma_1) (\sigma_\ell \sigma_{\ell-1} \cdots \sigma_2) \cdots (\sigma_\ell \sigma_{\ell-1}) (\sigma_\ell)(\lambda_j) \; = \; \cdots \\ \notag
& = & (\sigma_\ell \sigma_{\ell-1} \cdots \sigma_1) (\sigma_\ell \sigma_{\ell-1} \cdots \sigma_2) \cdots (\sigma_\ell \sigma_{\ell-1} \cdots \sigma_j)(\lambda_j) \\ \notag
& = & (\sigma_\ell \sigma_{\ell-1} \cdots \sigma_1) (\sigma_\ell \sigma_{\ell-1} \cdots \sigma_2) \cdots (\sigma_\ell \sigma_{\ell-1} \cdots \sigma_j-1)(\lambda_{j-1}+\lambda_{\ell-1}-\lambda_\ell) \\ \notag
& = & (\sigma_\ell \sigma_{\ell-1} \cdots \sigma_1) (\sigma_\ell \sigma_{\ell-1} \cdots \sigma_2) \cdots (\sigma_\ell \sigma_{\ell-1} \cdots \sigma_j-2)(\lambda_{j-2}+\lambda_{\ell-2}-\lambda_\ell) \; = \; \cdots \\ \notag
& = & (\sigma_\ell \sigma_{\ell-1} \cdots \sigma_1)(\lambda_1+\lambda_{\ell-(j-1)}-\lambda_\ell) \; = \; \lambda_{\ell-j} - \lambda_\ell
\end{eqnarray}
For $\lambda_\ell$, we have:
\begin{eqnarray} \notag
\sigma_{C}^{(\ell)}(\lambda_\ell) & = &  (\sigma_\ell \sigma_{\ell-1} \cdots \sigma_1) (\sigma_\ell \sigma_{\ell-1} \cdots \sigma_2) \cdots (\sigma_\ell \sigma_{\ell-1}) (\sigma_\ell)(\lambda_\ell) \\ \notag
& = & (\sigma_\ell \sigma_{\ell-1} \cdots \sigma_1)(\sigma_\ell \sigma_{\ell-1} \cdots \sigma_2) \cdots (\sigma_\ell \sigma_{\ell-1})(2\lambda_{\ell - 1} - \lambda_\ell)\\ \notag
& = & (\sigma_\ell \sigma_{\ell-1} \cdots \sigma_1)(\sigma_\ell \sigma_{\ell-1} \cdots \sigma_2) \cdots (\sigma_\ell \sigma_{\ell-1} \sigma_{\ell-2})
(2\lambda_{\ell - 2} - \lambda_\ell) \; = \; \cdots \\ \notag
& = & \sigma_\ell \sigma_{\ell-1} \cdots \sigma_1 (2\lambda_1 - \lambda_\ell)
\; = \; 2\lambda_{\ell-1} - 2\lambda_\ell - 2\lambda_{\ell-1} + \lambda_\ell 
\; = \; -\lambda_\ell.
\end{eqnarray}
Adhering to our convention (i.e. $\lambda_0=0$), we have that 
\begin{equation} \label{caction}
   \sigma_{C}^{(\ell)}(\lambda_j) = \lambda_{\ell-j} - \lambda_\ell
\end{equation}
for all $j=1,\dots,\ell$.

A quick calculation now shows that 
\begin{eqnarray*}
\sigma_{C}^{(\ell)}(\alpha_0) & = & \sigma_{C}^{(\ell)}(-\theta) \; = \; \alpha_\ell \\
\sigma_{C}^{(\ell)}(\alpha_j) & = & \alpha_{\ell-j} \quad (\mathrm{for}\;  1 \le j < \ell)\\
\sigma_{C}^{(\ell)}(\alpha_\ell) & = & -\theta \; = \; \alpha_0
\end{eqnarray*}

\subsection{Type $D_\ell$ (any rank)}

Looking at the Cartan matrix of type $D_\ell$, we see that
$\alpha_i = -\lambda_{i-1} + 2\lambda_i - \lambda_{i+1}$ for $1 \le i < \ell-2$,
$\alpha_{\ell-2} = -\lambda_{\ell-3}+2\lambda_{\ell-2}-\lambda_{\ell-1}-\lambda_\ell$,
$\alpha_{\ell-1} = -\lambda_{\ell-2}+2\lambda_{\ell-1}$, and
$\alpha_\ell = -\lambda_{\ell-2}+2\lambda_\ell$.

First, recall that $\sigma_{D}^{(1)} = \sigma_1 \sigma_2 \cdots \sigma_\ell \sigma_{\ell-2}
\sigma_{\ell-3} \cdots \sigma_2 \sigma_1$.

As a first step, we determine that action of $\sigma_1\sigma_2\cdots\sigma_\ell$ on
$\lambda_j$. For $1 \leq j < \ell-2$ we have:
\begin{eqnarray*}
\sigma_1\sigma_2\cdots\sigma_\ell(\lambda_j) 
& = & \sigma_1\sigma_2\cdots\sigma_j(\lambda_j) 
\; = \; \sigma_1\sigma_2\cdots\sigma_{j-1}(\lambda_{j-1}-\lambda_j+\lambda_{j+1}) 
\; = \; \cdots \\
& = & \sigma_1(\lambda_1-\lambda_2+\lambda_{j+1}) 
\; = \; -\lambda_1+\lambda_{j+1}
\end{eqnarray*}
\begin{eqnarray*}
\sigma_1\sigma_2\cdots\sigma_\ell(\lambda_{\ell-2}) 
& = & \sigma_1\sigma_2\cdots\sigma_{\ell-2}(\lambda_{\ell-2})
\; = \; \sigma_1\sigma_2\cdots\sigma_{\ell-3} (\lambda_{\ell-3}-\lambda_{\ell-2}+\lambda_{\ell-1}+\lambda_\ell) \\
& = & \sigma_1\sigma_2\cdots\sigma_{\ell-4} (\lambda_{\ell-4}-\lambda_{\ell-3}+\lambda_{\ell-1}+\lambda_\ell) \; = \; \cdots \\
& = & \sigma_1(\lambda_1-\lambda_2+\lambda_{\ell-1}+\lambda_\ell)
\; = \; -\lambda_1+\lambda_{\ell-1}+\lambda_\ell
\end{eqnarray*}
\begin{eqnarray*}
\sigma_1\sigma_2\cdots\sigma_\ell(\lambda_{\ell-1}) 
& = & \sigma_1\sigma_2\cdots\sigma_{\ell-1}(\lambda_{\ell-1})
\; = \; \sigma_1\sigma_2\cdots\sigma_{\ell-2}(\lambda_{\ell-2}-\lambda_{\ell-1}) \\
& = & \sigma_1\sigma_2\cdots\sigma_{\ell-3}(\lambda_{\ell-3}-\lambda_{\ell-2}+\lambda_\ell) 
\; = \; \sigma_1\sigma_2\cdots\sigma_{\ell-4} (\lambda_{\ell-4}-\lambda_{\ell-3}+\lambda_\ell) 
\\ & = & \cdots 
\; = \; \sigma_1(\lambda_1-\lambda_2+\lambda_\ell)
\; = \; -\lambda_1+\lambda_\ell
\end{eqnarray*}
\begin{eqnarray*}
\sigma_1\sigma_2\cdots\sigma_\ell(\lambda_\ell) 
& = & \sigma_1\sigma_2\cdots\sigma_{\ell-1}(\lambda_{\ell-2}-\lambda_\ell)
\; = \; \sigma_1\sigma_2\cdots\sigma_{\ell-2}(\lambda_{\ell-2}-\lambda_\ell) \\
& = & \sigma_1\sigma_2\cdots\sigma_{\ell-3}(\lambda_{\ell-3}-\lambda_{\ell-2}+\lambda_{\ell-1}) 
\; = \; \sigma_1\sigma_2\cdots\sigma_{\ell-4} (\lambda_{\ell-4}-\lambda_{\ell-3}+\lambda_{\ell-1}) 
\\ & = & \cdots 
\; = \; \sigma_1(\lambda_1-\lambda_2+\lambda_{\ell-1})
\; = \; -\lambda_1+\lambda_{\ell-1}
\end{eqnarray*}

For $1 \le j \le \ell-2$ we have:
\begin{eqnarray} \notag
 \sigma_{D}^{(1)}(\lambda_j) & = & \sigma_1 \sigma_2 \cdots \sigma_\ell \sigma_{\ell-2} \sigma_{\ell-3} \cdots \sigma_2 \sigma_1(\lambda_j)
\; = \; \sigma_1 \sigma_2 \cdots \sigma_\ell \sigma_{\ell-2} \cdots \sigma_j (\lambda_j) \\ \notag
& = & \sigma_1 \sigma_2 \cdots \sigma_\ell \sigma_{\ell-2} \cdots \sigma_{j+1} (\lambda_{j-1}-\lambda_j+\lambda_{j+1}) \\ \notag
& = & \sigma_1 \sigma_2 \cdots \sigma_\ell \sigma_{\ell-2} \cdots \sigma_{j+2} (\lambda_{j-1}-\lambda_{j+1}+\lambda_{j+2}) \; = \; \cdots \\ \notag
& = & \sigma_1 \sigma_2 \cdots \sigma_\ell \sigma_{\ell-2} (\lambda_{j-1}-\lambda_{\ell-3}+\lambda_{\ell-2}) \\ \notag
& = & \sigma_1 \sigma_2 \cdots \sigma_\ell (\lambda_{j-1}-\lambda_{\ell-2}+\lambda_{\ell-1}+\lambda\ell) \\ \notag
& = & (-\lambda_1+\lambda_j) - (-\lambda_1+\lambda_{\ell-1}+\lambda_\ell) +
      (-\lambda_1+\lambda_\ell) + (-\lambda_1+\lambda_{\ell-1}) \\
& = & -2\lambda_1+\lambda_j \label{daction1j}
\end{eqnarray}
\begin{equation}
\sigma_{D}^{(1)}(\lambda_{\ell-1}) \; = \; \sigma_1 \sigma_2 \cdots \sigma_\ell \sigma_{\ell-2} \sigma_{\ell-3} \cdots \sigma_2 \sigma_1(\lambda_{\ell-1}) 
\; = \; \sigma_1 \sigma_2 \cdots \sigma_\ell (\lambda_{\ell-1}) 
\; = \; -\lambda_1+\lambda_\ell \label{daction1ell-1}
\end{equation}
\begin{equation}
\sigma_{D}^{(1)}(\lambda_\ell) \; = \; \sigma_1 \sigma_2 \cdots \sigma_\ell \sigma_{\ell-2} \sigma_{\ell-3} \cdots \sigma_2 \sigma_1(\lambda_\ell) 
\; = \; \sigma_1 \sigma_2 \cdots \sigma_\ell (\lambda_\ell) 
\; = \; -\lambda_1+\lambda_{\ell-1} \label{daction1ell}
\end{equation}

A quick calculation shows:
\begin{eqnarray*}
\sigma_{D}^{(1)}(\alpha_0) & = & \sigma_{D}^{(1)}(-\theta) \; = \; \alpha_1 \\
\sigma_{D}^{(1)}(\alpha_1) & = & -\theta \; = \; \alpha_0 \\
\sigma_{D}^{(1)}(\alpha_j) & = & \alpha_j \quad (\mathrm{for}\; 2\le j \le \ell-2) \\
\sigma_{D}^{(1)}(\alpha_{\ell-1}) & = & \alpha_\ell \\
\sigma_{D}^{(1)}(\alpha_\ell) & = & \alpha_{\ell-1}.
\end{eqnarray*}

To help determine the actions of $\sigma_{D}^{(\ell-1)}$ and $\sigma_{D}^{(\ell)}$
on each $\lambda_j$, we will first consider $\sigma_{\ell-2}\sigma_{\ell-3}\cdots\sigma_1$.
For $1 \leq j \leq \ell-2$ we have:
\begin{eqnarray*}
\sigma_{\ell-2}\sigma_{\ell-3}\cdots\sigma_1(\lambda_j) 
& = & \sigma_{\ell-2}\sigma_{\ell-3}\cdots\sigma_j(\lambda_j) 
\; = \; \sigma_{\ell-2}\sigma_{\ell-3}\cdots\sigma_{j+1} (\lambda_{j-1}-\lambda_j+\lambda_{j+1}) \\
& = & \sigma_{\ell-2}\sigma_{\ell-3}\cdots\sigma_{j+2} (\lambda_{j-1}-\lambda_{j+1}+\lambda_{j+2}) \; = \; \cdots \\
& = & \sigma_{\ell-2}(\lambda_{j-1}-\lambda_{\ell-3}+\lambda_{\ell-2}) 
\; = \; \lambda_{j-1}-\lambda_{\ell-2}+\lambda_{\ell-1}+\lambda_\ell
\end{eqnarray*}
\[ \sigma_{\ell-2}\sigma_{\ell-3}\cdots\sigma_1(\lambda_{\ell-1}) = \lambda_{\ell-1} 
   \qquad \mathrm{and} \qquad
   \sigma_{\ell-2}\sigma_{\ell-3}\cdots\sigma_1(\lambda_\ell) = \lambda_\ell \]
Next, we consider $\sigma_{\ell-1}\sigma_{\ell-2}\sigma_{\ell-3}\cdots\sigma_1$.
For $1 \leq j \leq \ell-2$, we have:
\begin{eqnarray} \label{Dell-1Sigmas}
\sigma_{\ell-1}\sigma_{\ell-2}\sigma_{\ell-3}\cdots\sigma_1(\lambda_j) 
& = & \sigma_{\ell-1}(\lambda_{j-1}-\lambda_{\ell-2}+\lambda_{\ell-1}+\lambda_\ell) 
\; = \; \lambda_{j-1}-\lambda_{\ell-1}+\lambda_\ell
\end{eqnarray}
\[ \sigma_{\ell-1}\sigma_{\ell-2}\sigma_{\ell-3}\cdots\sigma_1(\lambda_{\ell-1}) =   
   \lambda_{\ell-2}-\lambda_{\ell-1} 
   \qquad \mathrm{and} \qquad
   \sigma_{\ell-1}\sigma_{\ell-2}\sigma_{\ell-3}\cdots\sigma_1(\lambda_\ell) = \lambda_\ell \]
Finally, consider $\sigma_\ell\sigma_{\ell-2}\sigma_{\ell-3}\cdots\sigma_1$.
For $1 \leq j \leq \ell-2$, we have:
\begin{eqnarray} \label{DellSigmas}
\sigma_\ell\sigma_{\ell-2}\sigma_{\ell-3}\cdots\sigma_1(\lambda_j) 
& = & \sigma_\ell(\lambda_{j-1}-\lambda_{\ell-2}+\lambda_{\ell-1}+\lambda_\ell) 
\; = \; \lambda_{j-1}+\lambda_{\ell-1}-\lambda_\ell
\end{eqnarray}
\[ \sigma_\ell\sigma_{\ell-2}\sigma_{\ell-3}\cdots\sigma_1(\lambda_{\ell-1}) = \lambda_{\ell-1} 
   \qquad \mathrm{and} \qquad
   \sigma_\ell\sigma_{\ell-2}\sigma_{\ell-3}\cdots\sigma_1(\lambda_\ell) = 
   \lambda_{\ell-2}-\lambda_\ell \]

\subsection{Type $D_{\ell}$ (odd rank)}

Now consider the case when $\ell$ is odd and recall that
\[ \sigma_{D}^{(\ell -1)} = (\sigma_{\ell-1}\sigma_{\ell-2}\cdots\sigma_1) (\sigma_\ell\sigma_{\ell-2}\sigma_{\ell-3}\cdots\sigma_2) (\sigma_{\ell-1}\sigma_{\ell-2}\cdots\sigma_3) \cdots (\sigma_{\ell}). \]

First, consider the case of $\lambda_1$.
\begin{eqnarray} \notag
 \sigma_{D}^{(\ell -1)}(\lambda_1) & = & (\sigma_{\ell-1}\sigma_{\ell-2}\cdots\sigma_1) (\sigma_\ell\sigma_{\ell-2}\sigma_{\ell-3}\cdots\sigma_2) (\sigma_{\ell-1}\sigma_{\ell-2}\cdots\sigma_3) \cdots (\sigma_{\ell})(\lambda_1)
\; = \; \cdots \\
& = & (\sigma_{\ell-1}\sigma_{\ell-2}\cdots\sigma_1)(\lambda_1)
\; = \; -\lambda_{\ell-1}+\lambda_\ell \label{dactionell-11}
\end{eqnarray}
(We obtain the last step by using (\ref{Dell-1Sigmas}) above.)

Next, consider the case of $\lambda_j$, $1 < j \leq \ell-2$ and $j$ odd.
By applying (\ref{Dell-1Sigmas}) and (\ref{DellSigmas}) successively, we obtain the following:
\begin{eqnarray*}
 \sigma_{D}^{(\ell-1)}(\lambda_j) & = & (\sigma_{\ell-1}\sigma_{\ell-2}\cdots\sigma_1) (\sigma_\ell\sigma_{\ell-2}\sigma_{\ell-3}\cdots\sigma_2) (\sigma_{\ell-1}\sigma_{\ell-2}\cdots\sigma_3) \cdots (\sigma_{\ell})(\lambda_j) 
\; = \; \cdots \\
 & = & (\sigma_{\ell-1}\sigma_{\ell-2}\cdots\sigma_1) \cdots (\sigma_{\ell-1}\sigma_{\ell-2}\cdots\sigma_j)(\lambda_j) \\
 & = & (\sigma_{\ell-1}\sigma_{\ell-2}\cdots\sigma_1) \cdots (\sigma_{\ell}\sigma_{\ell-2}\cdots\sigma_{j-1})(\lambda_{j-1}-\lambda_{\ell-1}+\lambda_\ell) \\
 & = & (\sigma_{\ell-1}\sigma_{\ell-2}\cdots\sigma_1) \cdots (\sigma_{\ell-1}\sigma_{\ell-2}\cdots\sigma_{j-2})(\lambda_{j-2}+\lambda_{\ell-2}-2\lambda_\ell) \\
 & = & (\sigma_{\ell-1}\sigma_{\ell-2}\cdots\sigma_1) \cdots (\sigma_{\ell}\sigma_{\ell-2}\cdots\sigma_{j-3})(\lambda_{j-3}+\lambda_{\ell-3}-2\lambda_{\ell-1}) 
\; = \; \cdots \\
 & = & (\sigma_{\ell-1}\sigma_{\ell-2}\cdots\sigma_1) (\lambda_{1}+\lambda_{\ell-(j-1)}-2\lambda_\ell) \\ 
& = & -\lambda_{\ell-1}+\lambda_\ell+\lambda_{\ell-j}-\lambda_{\ell-1}+\lambda_\ell-2\lambda_\ell\\
 & = & \lambda_{\ell-j}-2\lambda_{\ell-1}
\end{eqnarray*}

A similar calculation shows that the same holds for $j$ even and $2 \leq j \leq \ell-3$. 
Therefore, we have that 
\begin{equation}
\sigma_{D}^{(\ell-1)}(\lambda_j) = \lambda_{\ell-j}-2\lambda_{\ell-1} \quad (\mathrm{for}\;1 < j \leq \ell-2) \label{dactionell-1j}
\end{equation}
This leaves the cases $j=\ell-1$ and $j=\ell$.
\begin{eqnarray} \notag
 \sigma_{D}^{(\ell-1)}(\lambda_{\ell-1}) & = & (\sigma_{\ell-1}\sigma_{\ell-2}\cdots\sigma_1) (\sigma_\ell\sigma_{\ell-2}\sigma_{\ell-3}\cdots\sigma_2) (\sigma_{\ell-1}\sigma_{\ell-2}\cdots\sigma_3) \cdots (\sigma_{\ell})(\lambda_{\ell-1})\\ \notag
& = & (\sigma_{\ell-1}\sigma_{\ell-2}\cdots\sigma_1) \cdots (\sigma_{\ell-1}\sigma_{\ell-2})(\lambda_{\ell-1})\\ \notag
& = & (\sigma_{\ell-1}\sigma_{\ell-2}\cdots\sigma_1) \cdots (\sigma_{\ell}\sigma_{\ell-2}\sigma_{\ell-3})(\lambda_{\ell-2}-\lambda_{\ell-1})\\ \notag
& = & (\sigma_{\ell-1}\sigma_{\ell-2}\cdots\sigma_1) \cdots (\sigma_{\ell-1}\sigma_{\ell-2}\cdots\sigma_{\ell-4})(\lambda_{\ell-3}-\lambda_\ell)
\; = \; \cdots \\ \notag
& = & (\sigma_{\ell-1}\sigma_{\ell-2}\cdots\sigma_1)(\lambda_2-\lambda_\ell)\\ 
& = & \lambda_1-\lambda_{\ell-1}+\lambda_\ell-\lambda_\ell \; = \; \lambda_1-\lambda_{\ell-1} \label{dactionell-1ell-1}
\end{eqnarray}
\begin{eqnarray} \notag
 \sigma_{D}^{(\ell-1)}(\lambda_\ell) & = & (\sigma_{\ell-1}\sigma_{\ell-2}\cdots\sigma_1) (\sigma_\ell\sigma_{\ell-2}\sigma_{\ell-3}\cdots\sigma_2) (\sigma_{\ell-1}\sigma_{\ell-2}\cdots\sigma_3) \cdots (\sigma_{\ell})(\lambda_\ell)\\ \notag
& = & (\sigma_{\ell-1}\sigma_{\ell-2}\cdots\sigma_1) \cdots (\sigma_{\ell-1}\sigma_{\ell-2})(\lambda_{\ell-2}-\lambda_\ell)\\ \notag
& = & (\sigma_{\ell-1}\sigma_{\ell-2}\cdots\sigma_1) \cdots (\sigma_{\ell}\sigma_{\ell-2}\sigma_{\ell-3})(\lambda_{\ell-3}-\lambda_{\ell-1})\\ \notag
& = & (\sigma_{\ell-1}\sigma_{\ell-2}\cdots\sigma_1) \cdots (\sigma_{\ell-1}\sigma_{\ell-2}\cdots\sigma_{\ell-4})(\lambda_{\ell-4}-\lambda_\ell)
\; = \; \cdots \\ \notag
& = & (\sigma_{\ell-1}\sigma_{\ell-2}\cdots\sigma_1)(\lambda_1-\lambda_\ell)\\
& = & -\lambda_{\ell-1}+\lambda_\ell-\lambda_\ell \; = \; -\lambda_{\ell-1} \label{dactionell-1ell}
\end{eqnarray}

A quick calculation shows:
\begin{eqnarray*}
\sigma_{D}^{(\ell-1)}(\alpha_0) & = & \sigma_{D}^{(\ell-1)}(-\theta) \; = \; \alpha_{\ell-1} \\
\sigma_{D}^{(\ell-1)}(\alpha_1) & = & \alpha_\ell\\
\sigma_{D}^{(\ell-1)}(\alpha_j) & = & \alpha_{\ell-j} \quad (\mathrm{for}\; 2\le j \le \ell-1) \\
\sigma_{D}^{(\ell-1)}(\alpha_\ell) & = & -\theta \; = \; \alpha_0.
\end{eqnarray*}

Using a Dynkin diagram symmetry, we see that all of these results should still hold
if we interchange the labels $\ell-1$ and $\ell$. So we also have that
\begin{eqnarray} \notag
\sigma_{D}^{(\ell)}(\lambda_1) & = & \lambda_{\ell-1} - \lambda_{\ell}\\ 
\sigma_{D}^{(\ell)}(\lambda_j) & = & \lambda_{\ell-j} - 2\lambda_\ell  \quad (\mathrm{for}\; 2\le j \le \ell-2) \\ \notag
\sigma_{D}^{(\ell)}(\lambda_{\ell-1}) & = & -\lambda_\ell \\ \notag
\sigma_{D}^{(\ell)}(\lambda_\ell) & = & \lambda_1-\lambda_\ell,
\end{eqnarray}
and also
\begin{eqnarray*}
\sigma_{D}^{(\ell)}(\alpha_0) & = &  \sigma_{D}^{(\ell)}(-\theta) \; = \; \alpha_\ell\\
\sigma_{D}^{(\ell)}(\alpha_j) & = & \alpha_{\ell-j} \quad (\mathrm{for}\; 1\le j \le \ell-2)\\
\sigma_{D}^{(\ell)}(\alpha_{\ell-1}) & = & -\theta \; = \; \alpha_0 \\
\sigma_{D}^{(\ell)}(\alpha_\ell) & = & \alpha_1.
\end{eqnarray*}

\subsection{Type $D_{\ell}$ (even rank)}

Almost identical calculations reveal that for even $\ell$ we have the following:
\begin{eqnarray} \notag
\sigma_{D}^{(\ell-1)}(\lambda_1) & = & -\lambda_{\ell-1} + \lambda_{\ell} \\ \label{dactioneven}
\sigma_{D}^{(\ell-1)}(\lambda_j) & = & \lambda_{\ell-j} - 2\lambda_{\ell-1} \quad (\mathrm{for}\; 2\le j \le \ell-2) \\ \notag
\sigma_{D}^{(\ell-1)}(\lambda_{\ell-1}) & = & -\lambda_{\ell-1} \\ \notag
\sigma_{D}^{(\ell-1)}(\lambda_\ell) & = & \lambda_1 - \lambda_{\ell-1} 
\end{eqnarray}
and it follows that:
\begin{eqnarray*}
\sigma_{D}^{(\ell-1)}(\alpha_0) & = & \sigma_{D}^{(\ell-1)}(-\theta) \; = \; \alpha_{\ell-1}\\ 
\sigma_{D}^{(\ell-1)}(\alpha_1) & = & \alpha_\ell\\
\sigma_{D}^{(\ell-1)}(\alpha_j) & = & \alpha_{\ell-j} \quad (\mathrm{for}\; 2\le j \le \ell-2) \\
\sigma_{D}^{(\ell-1)}(\alpha_{\ell-1}) & = & -\theta \; = \; \alpha_0\\
\sigma_{D}^{(\ell-1)}(\alpha_\ell) &=& \alpha_1.
\end{eqnarray*}

Again using a Dynkin diagram symmetry, we see that all of these results should still hold
if we interchange the labels $\ell-1$ and $\ell$. So we also have that:
\begin{eqnarray} \notag
\sigma_{D}^{(\ell)}(\lambda_1) & = & \lambda_{\ell-1} - \lambda_{\ell} \\
\sigma_{D}^{(\ell)}(\lambda_j) & = & \lambda_{\ell-j} - 2\lambda_\ell \quad (\mathrm{for}\; 2\le j \le \ell-2) \\ \notag
\sigma_{D}^{(\ell)}(\lambda_{\ell-1}) & = & \lambda_1 - \lambda_\ell \\ \notag
\sigma_{D}^{(\ell)}(\lambda_\ell) & = & - \lambda_\ell
\end{eqnarray}
and it follows that (recall $-\theta=\alpha_0$):
\begin{eqnarray*}
\sigma_{D}^{(\ell)}(\alpha_j) & = & \alpha_{\ell-j} \quad (\mathrm{for}\; 1\le j \le\ell).
\end{eqnarray*}

\subsection{Exceptional Types}

Recall that if $\g$ is of type $E_8$, $F_4$ or $G_2$, then $\g$ has no miniscule weights
and so $P^\vee=Q^\vee$. Thus the action of $\Delta(H,x)$ is always trivial. So we only 
need to consider $\g$ of type $E_6$ and $E_7$. 
 
The following calculations for types $E_6$ and $E_7$ were done with the help of a Maple worksheet which is available at:\\ \verb|http://dimax.rutgers.edu/~sadowski/LieAlgebraCalculations/index.html|

For type $E_6$ using $H^{(1)}$, we have:
\begin{eqnarray} \notag
\sigma_{E}^{(1)}(\lambda_1) &=& -\lambda_1 + \lambda_6\\ \notag
\sigma_{E}^{(1)}(\lambda_2) &=& -2\lambda_1 + \lambda_3\\ \label{eaction1}
\sigma_{E}^{(1)}(\lambda_3) &=& -2\lambda_1 + \lambda_5\\ \notag
\sigma_{E}^{(1)}(\lambda_4) &=& -3\lambda_1 + \lambda_4\\ \notag
\sigma_{E}^{(1)}(\lambda_5) &=& -2\lambda_1 + \lambda_2\\ \notag
\sigma_{E}^{(1)}(\lambda_6) &=& -\lambda_1
\end{eqnarray}
and so
\begin{eqnarray*}
\sigma_{E}^{(1)}(\alpha_0) & = & \sigma_{E}^{(1)}(-\theta) = \alpha_1\\
\sigma_{E}^{(1)}(\alpha_1) & = & \alpha_6 \\
\sigma_{E}^{(1)}(\alpha_2) & = & \alpha_3 \\
\sigma_{E}^{(1)}(\alpha_3) & = & \alpha_5 \\
\sigma_{E}^{(1)}(\alpha_4) & = & \alpha_4 \\
\sigma_{E}^{(1)}(\alpha_5) & = & \alpha_2 \\
\sigma_{E}^{(1)}(\alpha_6) & = & -\theta \; = \; \alpha_0.
\end{eqnarray*}

For type $E_6$ using $H^{(6)}$, we have:
\begin{eqnarray} \notag
\sigma_{E}^{(6)}(\lambda_1) &=& -\lambda_6\\ \notag
\sigma_{E}^{(6)}(\lambda_2) &=& -2\lambda_6 + \lambda_5\\ \label{eaction6}
\sigma_{E}^{(6)}(\lambda_3) &=& -2\lambda_6 + \lambda_2\\ \notag
\sigma_{E}^{(6)}(\lambda_4) &=& -3\lambda_6 + \lambda_4\\ \notag
\sigma_{E}^{(6)}(\lambda_5) &=& -2\lambda_6 + \lambda_3\\ \notag
\sigma_{E}^{(6)}(\lambda_6) &=& -\lambda_6 + \lambda_1
\end{eqnarray}
and so
\begin{eqnarray*}
\sigma_{E}^{(6)}(\alpha_0) & = & \sigma_{E}^{(6)}(-\theta) \; = \; \alpha_6\\
\sigma_{E}^{(6)}(\alpha_1) & = & \alpha_5 \\
\sigma_{E}^{(6)}(\alpha_2) & = & \alpha_4 \\
\sigma_{E}^{(6)}(\alpha_3) & = & \alpha_3 \\
\sigma_{E}^{(6)}(\alpha_4) & = & \alpha_2 \\
\sigma_{E}^{(6)}(\alpha_5) & = & \alpha_1 \\
\sigma_{E}^{(6)}(\alpha_6) & = & -\theta \; = \; \alpha_0.
\end{eqnarray*}

For type $E_7$ using $H^{(7)}$, we have:
\begin{eqnarray} \notag
\sigma_{E}^{(7)}(\lambda_1) &=& -2\lambda_7 + \lambda_6 \\ \notag
\sigma_{E}^{(7)}(\lambda_2) &=& -2\lambda_7 + \lambda_2 \\ \notag
\sigma_{E}^{(7)}(\lambda_3) &=& -3\lambda_7 + \lambda_5 \\ \label{eaction7}
\sigma_{E}^{(7)}(\lambda_4) &=& -4\lambda_7 + \lambda_4\\ \notag
\sigma_{E}^{(7)}(\lambda_5) &=& -3\lambda_7 + \lambda_3\\ \notag
\sigma_{E}^{(7)}(\lambda_6) &=& -2\lambda_7 + \lambda_1\\ \notag
\sigma_{E}^{(7)}(\lambda_7) &=& -\lambda_7
\end{eqnarray}
and so
\begin{eqnarray*}
\sigma_{E}^{(7)}(\alpha_0) & = & \sigma_{E}^{(7)}(-\theta) \; = \; \alpha_7\\
\sigma_{E}^{(7)}(\alpha_1) & = & \alpha_6 \\
\sigma_{E}^{(7)}(\alpha_2) & = & \alpha_2 \\
\sigma_{E}^{(7)}(\alpha_3) & = & \alpha_5 \\
\sigma_{E}^{(7)}(\alpha_4) & = & \alpha_4 \\
\sigma_{E}^{(7)}(\alpha_5) & = & \alpha_3 \\
\sigma_{E}^{(7)}(\alpha_6) & = & \alpha_1 \\
\sigma_{E}^{(7)}(\alpha_7) & = & -\theta \; = \; \alpha_0.
\end{eqnarray*}


\begin{thebibliography}{CookSad-bib}
\bibitem[C]{c-thesis}
W. Cook, Affine Lie algebras, vertex operator algebras, and combinatorial identities,
Ph.D. thesis, North Carolina State University, 2005.

\bibitem[CLM]{clm}
W. Cook, H.-S. Li and K. C. Misra, A recurrence relation for characters of highest weight
integrable modules for affine {L}ie algebras, {\em Commun. Contemp. Math.} {\bf 9} (2007),
no. 2, 121-133.

\bibitem[CMS]{cms}
W. Cook, C. Mitschi and M. F. Singer, On the constructive inverse problem in differential {G}alois theory, {\em Comm. Algebra} {\bf 33} (2005), 3639-3665.

\bibitem[DL]{dl}
C. Dong and J. Lepowsky, {\it Generalized Vertex Algebras and Relative Vertex Operators}, 
Progress in Math., {\bf Vol. 112}, Birkh\"auser, Boston, 1993.

\bibitem[DLM]{dlm}
C. Dong, H.-S. Li and G. Mason, 
Regularity of rational vertex operator algebras,
{\em Adv. Math.} {\bf 132} (1997), 148-166.

\bibitem[FHL]{fhl}
I. Frenkel, Y.-Z. Huang and J. Lepowsky, On axiomatic approaches to
vertex operator algebras and modules, Memoirs Amer. Math. Soc. {\bf 104}, 1993.

\bibitem[FLM]{flm}
I. Frenkel, J. Lepowsky and A. Meurman, {\it Vertex Operator Algebras and the Monster}, 
Pure and Appl. Math., {\bf Vol. 134}, Academic Press, Boston, 1988.

\bibitem[H]{hum}
J. Humphreys, {\em Introduction to Lie Algebras and Their Representations,}
Springer-Verlag, New York-Heildelberg-Berlin, 1972.

\bibitem[K]{kac-book}
V. G. Kac, {\it Infinite Dimensional Lie Algebras},
Cambridge University Press, 3rd edition, 1990.

\bibitem[LL]{ll-book}
J.  Lepowsky and H.-S. Li, {\it Introduction to Vertex Operator
Algebras and Their Representation Theory},  Progress in Math., {\bf Vol. 227}, 
Birkh\"{a}user, Boston, 2004.

%\bibitem[Li1]{li-form}
%H.-S. Li, Symmetric invariant bilinear forms on vertex operator
%algebras, {\em J. Pure Appl. Alg.} {\bf 96} (1994), 279-297.

\bibitem[Li1]{li-local}
H.-S. Li, Local systems of vertex operators, vertex superalgebras and
modules, {\em J. Pure Appl. Alg.} {\bf 109} (1996), 143-195.

%\bibitem[Li2]{li-super}
%H.-S. Li, The physics superselection principle in vertex operator
%algebra theory, {\em J. Algebra} {\bf 196} (1997), 436-457.

%changed from [Li3]
\bibitem[Li2]{li-affine}
H.-S. Li, Certain extensions of vertex operator algebras of affine type, 
{\em Commun. Math. Phys.} {\bf 217} (2001), 653-696.

\end{thebibliography}
\end{document}